\documentclass[a4paper]{amsproc}     

\usepackage{graphicx}
\usepackage[labelfont=bf, labelsep=space]{caption}
\usepackage{amsmath, amssymb, amsfonts, textcomp, mathrsfs, float, ulem, subcaption}
\newtheorem{thm}{Theorem}
\newtheorem{lem}{Lemma}

\newtheorem{rem}{Remark}
\newtheorem{cor}{Corollary}
\newtheorem{prop}{Proposition}
\newtheorem{df}{Definition}
\newtheorem{ques}{Question}

\newtheorem{manualtheoreminner}{Theorem}
\newenvironment{manualtheorem}[1]{%
  \renewcommand\themanualtheoreminner{#1}%
  \manualtheoreminner
}{\endmanualtheoreminner}

\newcommand{\taf}{T_{a_4}(a_0)}

\newcommand{\cc}{\mathcal{C}(S_g)}

\begin{document}

\title{Distance $5$ Curves in the Curve Graph of Closed Surfaces}

\author{Kuwari Mahanta} 
\maketitle

\begin{abstract}
Let $S_g$ denote a closed, orientable surface of genus $g \geq 2$ and $\cc$ be the associated curve graph. Let $d$ be the path metric on $\cc$ and $a_0$ and $a_4$ be a pair of curves on $S_g$ with $d(a_0, a_4) = 4$. In this article, we fix the vertex $a_0$ and apply the Dehn twist about $a_4$, $T_{a_4}$, to it in an attempt to create pairs of curves at a distance $5$ apart. We give a necessary and sufficient topological condition for $d(a_0, \taf)$ to be $4$. We then characterise the pairs of $a_0$ and $a_4$ for which $5 \leq d(a_0, \taf) \leq 6$. Lastly, we give an example of a pair of curves on $S_2$ which represent vertices at a distance $5$ in $\mathcal{C}(S_2)$ with intersection number $144$.
\end{abstract}

\section{Introduction}
Let $S_g$ be a closed, orientable surface of genus $g \geq 2$. By a curve, $\gamma$, on $S_g$, we mean an embedding of the unit circle in $S_g$ such that the image isn't null-homotopic. In \cite{H1}, William J. Harvey introduced a finite dimensional simplical complex corresponding to $S_g$, called the complex of curves of $S_g$, as a tool to study the Teichmuller spaces of Riemann surfaces. Our work focuses on the $1$-skeleton of the curve complex of $S_g$, called the curve graph. The curve graph of $S_g$, denoted by $\cc$, is defined as follows : the set of vertices of $\cc$ comprises of isotopy classes of curves on $S_g$ and two vertices in $\cc$ share an edge if they have disjoint representatives. As $\cc$ is connected, a path-metric can be put on it by defining the distance, $d$, between any two vertices as the minimal number of edges in any edge path between them. For a curve $\gamma$ on $S_g$, by an excusable abuse of notation, we will use $\gamma$ to represent both the curve $\gamma$ as well as its isotopy class. For any collection of curves on $S_g$, we will consider representatives which are in minimal position with each other. The intersection number between any two curves is the least number of intersections between any two curves in their respective isotopy classes. For curves $\alpha$ and $\beta$ on $S_g$ we denote their intersection number by $i(\alpha, \beta)$. For $n \in \mathbb{N}$, we denote the minimal intersection number of the set of curves on $S_g$ which are at a distance $n$ apart as $i_{min}(g,n)$.

The coarse geometry of the curve complex plays a pivotal role in understanding not only the Teichm\"uller theory but also the hyperbolic structures of $3$-manifolds and the mapping class group of surfaces. A seminal and pioneering work in the direction of the exploration of this coarse geometry are the articles \cite{MM1, MM2} by Masur and Minsky. In comparison, the local geometry of $\cc$ remains relatively unexplored.  We attempt a study in this direction by looking at the impact of low powers of Dehn twists on vertices of $\cc$ at shorter distances apart. Let $a_0$, $a_1$, $a_2$,\dots, $a_n$ for $n \in \mathbb{N}$ be a geodesic of length $n$ in $\mathcal{C}(S_g)$. Let $k \in \mathbb{N}$. For $n= 1, 2$, it is an easy observation that $d(a_0, T_{a_n}^k(a_0)) = d(a_0, a_n) - 1$ and $d(a_0, T_{a_n}^k(a_0)) = d(a_0, a_n)$, respectively. From \cite{chicken}, we have that for $n=3$, $d(a_0, T_{a_n}^k(a_0)) = d(a_0, a_n) + 1$. We are motivated to look at this mechanism with the long term promise of creating examples of curves at a distance $n+1$ by using curves at a distance $n$ apart. Workable examples of pairs of curves which are at a distance $n \geq 5$ apart on $\mathcal{C}(S_g)$ with low intersection number are not known in the literature.

In this article, we first show that $4 \leq d(a_0, \taf) \leq 6$. In section \ref{sec:set}, we define a family of curves that fill $S_g$ along with $a_0$ which we call as scaling curves. The idea behind a scaling curve is that $a_4$ encodes the information of a few naturally occurring curves which are at a distance $3$ from $a_0$ and distance $1$ from $a_4$. We employ these scaling curves to give a necessary and sufficient condition for $d(a_0, \taf) =4$ as stated in lemma \ref{lem:geq4}. 

Let $N$ be an annular neighbourhood of $a_4$ and $B_m(\taf)$ be the sphere of radius $m$ around $\taf$. Let $\delta \in B_1(\taf)$ and $c \in B_2(\taf) \cap B_1(\delta)$. Then, $c$ is a standard single strand curve if $i(c, a_4) = 1$ and if there exists an isotopic representative of $c$ such that $(c \cap a_0) \subset N$. In section \ref{sec:conc}, we describe a placement of the components of $N \setminus (a_0 \cup a_4)$ which is equivalent to there being a curve on $S_g$ which is mutually disjoint from $a_0$ and $c$. We call this arrangement of the components as the stacking property. We then have the following theorem which gives that $5 \leq d(a_0, \taf) \leq 6$ for a judicious choice of $a_0$ and $a_4$.

\begin{manualtheorem}{\ref{thm:geq5}}
Let $a_0$ and $a_4$ be curves on $S_g$ such that $d(a_0, a_4) = 4$ and the components of $S_g \setminus (a_0 \cup a_4)$ doesn't contain any hexagons. Then, $d(a_0, \taf) \geq 5$ if and only if there doesn't exist any standard single strand curve $c \in B_2(\taf)$ having the stacking property.
\end{manualtheorem} 

Let $b_0$ and $b_4$ be the pair of minimally intersecting curves at a distance $4$ on $\mathcal{C}(S_2)$ as given in \cite{BMM}. In section \ref{sec:eg5}, we show that $d(b_0, T_{b_4}(b_0)) = 5$ and give a geodesic between $b_0$ and $T_{b_4}(b_0)$. We note that contrary to the hypothesis of theorem \ref{thm:geq5}, there exists components of $S_g \setminus (b_0 \cup b_4)$ which are hexagons. An immediate conclusion from this example is that $i_{min}(2,5) \leq 144$. Further, this gives us a rendering of an example of a pair of distance $5$ curves. 

Using arguments involving subsurface projection, the authors in \cite{AT1} show that $d(a_0, T^B_{T^B_{a_3}(a_0)}(a_0)) \geq 6$ for a large enough constant $B \in \mathbb{N}$. The same arguments can be used to show that for any $a_4$ there exists a constant $K \in \mathbb{N}$ such that $d(a_0, T_{a_4}^k(a_0)) \geq 6$ for every $k \geq K$. We conclude that unlike $n \leq 3$, for $n=4$ we have that $d(b_0, T_{b_4}(b_0)) = d(b_0, b_4) + 1$ whereas, $d(b_0, T_{b_4}^k(b_0)) = d(b_0, b_4) + 2$, $\forall k \geq K$ where $K$ is some large enough constant. Thus, when $n\leq 3$ the value of $d(a_0, T_{a_n}^k(a_0))$ is independent of $k$ whereas, when $n=4$, this value depends on $k$. We conjecture that $d(a_0, \taf) =5$ only if $a_0$ and $a_4$ are minimally intersecting.

The above information thus prompts us to ask the following questions :
\begin{ques}
Given a pair of curves $a_0$ and $a_4$ on $S_g$ with $d(a_0, a_4)=4$, what are the values of $k \in \mathbb{N}$ such that $d(a_0, T_{a_4}^k(a_0)) = 6$?
\end{ques}

\begin{ques}
Is $i_{min}(g, 5) \leq i_{min}(g,4)^2$?
\end{ques}

In \cite{MM2}, Masur and Minsky proved that the curve complex corresponding to $S_g$ is $\delta$-hyperbolic. Later, the authors in \cite{A1, B1, CRS1, HPW1} independently showed that a uniform $\delta$ exists. From \cite{chicken}, we note that the geodesic triangle in $\cc$ with vertices $a_0$, $a_3$ and $T_{a_3}^k(a_0)$ is $0$-hyperbolic for every $k \in \mathbb{N}$. The results in \cite{AT1} give that the geodesic triangle in $\cc$ with vertices $a_0$, $a_3$ and $T^B_{T^B_{a_3}(a_0)}(a_0)$ is $0$-hyperbolic for some constant $B \in \mathbb{N}$. As a consequence of the aforementioned example in section \ref{sec:eg5} of the pair of curves, $b_0$ and $T_{b_4}(b_0)$, we observe that the geodesic triangle formed with vertices $b_0$, $b_4$ and $T_{b_4}(b_0)$ is $1$-hyperbolic. This leads to the conclusion that geodesic triangles formed with vertices $a_0$, $a_n$ and $T_{a_n}^k(a_0)$ need not be $0$-hyperbolic for all values of $k \in \mathbb{N}$. This leads to the following prospective question that was suggested by Joan Birman: 
\begin{ques}
Let $a_0$ and $a_n$ be a pair of curves on $S_g$ such that $d(a_0, a_n) = n$. Let $\Delta$ be the geodesic triangle in $\cc$ with vertices $a_0$, $a_n$ and $T_{a_n}(a_0)$. What is the minimum value of $\delta$ for which $\Delta$ is $\delta$-hyperbolic?
\end{ques}


\section{Acknowldgement}
The author would like to thank Sreekrishna Palaparthi for his invaluable comments and carefully reading through the drafts of this manuscript. The author is grateful to Joan Birman for her insightful remarks that helped further the scope of this work.


\section{Setup}
\label{sec:set}
In this section for a given filling pair of curves on $S_g$, we describe isotopic representatives of these curves and select corresponding regular neighbourhoods such that they satisfy certain favourable conditions. We further give terminology which aids the discourse of the proof of theorem \ref{thm:geq5}. Lastly we introduce a few techniques in the form of lemmas that will be used in section \ref{sec:main}.

In this article, for any ordered index, we follow cyclical ordering. For instance, if $i \in \{1, 2, \dots, m\}$, $i=m+1$ will indicate $i=1$. For any collection of curves on $S_g$, we will consider their isotopic representations which are in minimal position with each other. Let $\alpha$ be a curve on $S_g$. We denote the sphere of radius $n \in \mathbb{N}$ in $\cc$ with centre $\alpha$ as $B_n(\alpha)$. We use $R_a$ to denote an regular neighbourhood of $\alpha$. Let $\partial(R_\alpha)$ denote the multicurve comprising of the two boundary curves of $R_\alpha$, each of which is isotopic to $\alpha$. The following definition \ref{def:amenable} describes isotopic representatives of a pair of filling curves on $S_g$ and a choice of their regular neighbourhoods. A proof to the existence of these representatives can be found in \cite{chicken}. 

\begin{df}\label{def:amenable}[Amenable to Dehn twist in special position]
Let $\lambda$ and $\mu$ be two curves on $S_g$ and let $R_\lambda$ and $R_\mu$ be closed regular neighbourhoods of $\lambda$ and $\mu$ respectively. We say that the $4$-tuple $(\lambda, \mu, R_\lambda, R_\mu)$ is \textit{amenable to Dehn twist in special position} if the following hold:
\begin{enumerate}
 \item $\lambda$ and $\mu$ intersect transversely and minimally on $S_g$,
 \item $\lambda$ and $\mu$ fill $S_g$,
 \item the number of components of $R_\lambda \cap R_\mu$ is equal to the intersection number of $\lambda$ and $\mu$ and each of these components is a disc.
 \item $\mu$ and $\lambda$ are in minimal position with the components of $\partial(R_\lambda)$ and $\partial(R_\mu)$, respectively.
\end{enumerate}
\end{df}

\begin{figure}
     \centering
     \begin{subfigure}{0.32\textwidth}
         \centering
	\includegraphics[width=\textwidth]{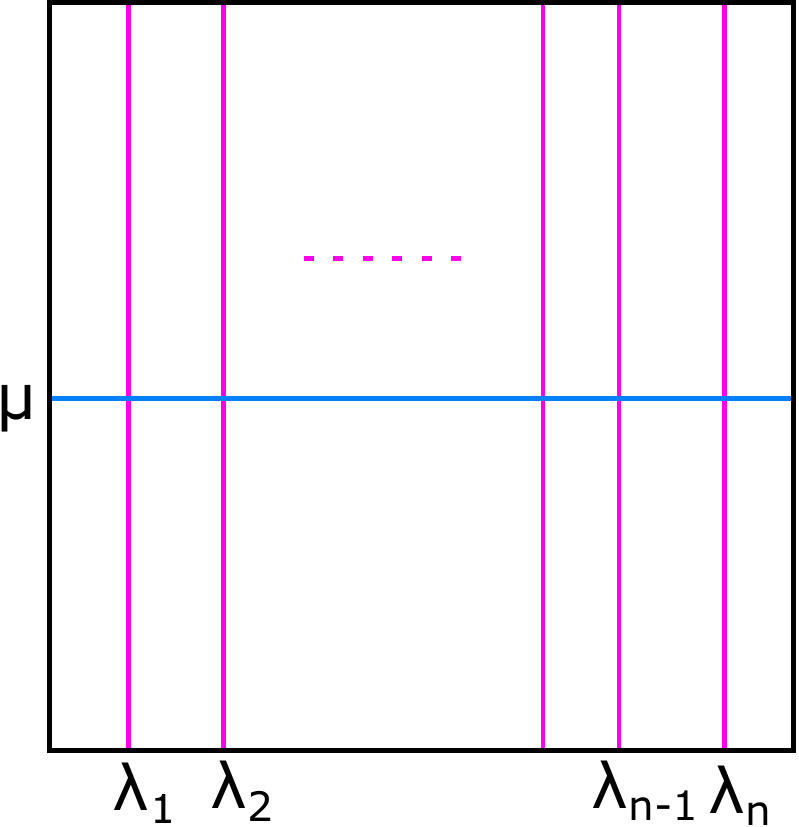}
     \end{subfigure}
     \hspace{2cm}
     \begin{subfigure}{0.3\textwidth}
         \centering
	\includegraphics[width=\textwidth]{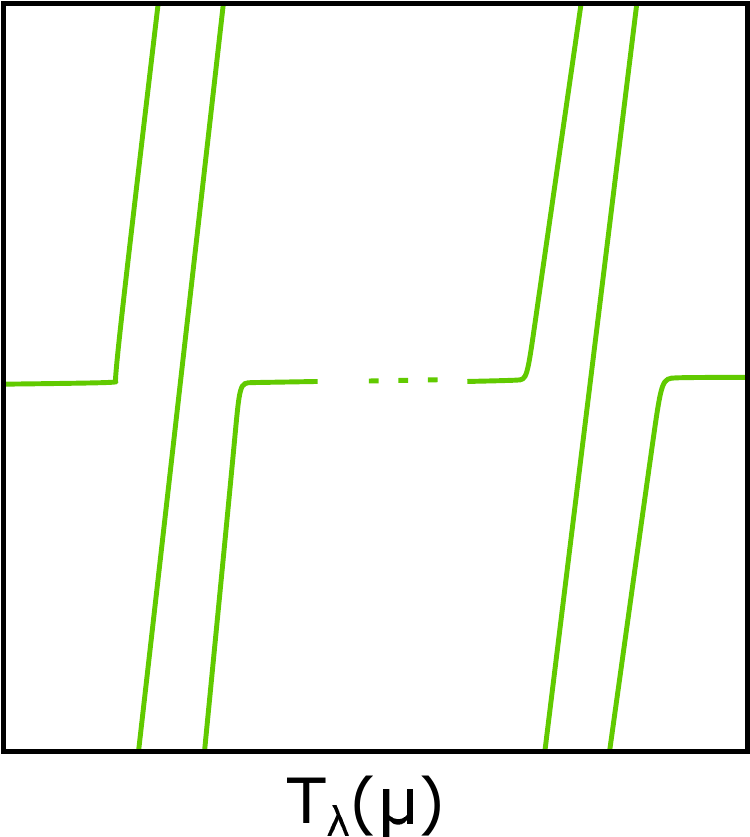}
     \end{subfigure}
	\caption{Surgery performed on $\mu$ and $\lambda_i$'s in $R_\lambda \cap 		R_\mu$}
	\label{img:DTsurgery}
\end{figure}

Let $(\lambda, \mu, R_{\lambda}, R_{\mu})$ be amenable to Dehn twist in special position and $n = i(\lambda, \mu)$. Consider $\lambda_1, \dots, \lambda_n$ be $n$- copies of  $\lambda$ such that they are in minimal position with each other along with $\lambda$, $\mu$, $\partial(R_{\lambda})$ and $\partial(R_{\mu})$. We perform a surgery as in figure \ref{img:DTsurgery} on the arcs of $\lambda_i$ in every component of $R_\lambda \cap R_\mu$. The resultant of this surgery is a simple closed curve on $S_g$ and is a representative of $T_\lambda(\mu)$ (see \cite{FM}) which is in minimal position with $\lambda$, $\mu$, $\partial(R_{\lambda})$ and $\partial(R_{\mu})$. We say \textit{$T_\lambda(\mu)$ is in special position w.r.t. $\lambda$ and $\mu$} to mean a representative of $T_\lambda(\mu)$ as obtained after this surgery. For details regarding special position of Dehn twist, see section $2$ of \cite{chicken}. 

$S_g$ can also be considered as a $2$ dimensional CW complex formed from the graph of $\mu \cup \lambda$ on $S_g$ in the following manner : The $0$ skeleton comprises of the $n$ distinct points in $\mu \cap \lambda$. The $1$-skeleton is formed by joining two vertices with an edge if and only if there is an arc of either $\mu$ or $\lambda$ between them. Attach a $2$ simplex to a cycle in the $1$ skeleton if and only if the arcs of $\mu$ and $\lambda$ corresponding to the edges of the cycle bound a disc of $S_g \setminus (\mu \cup \lambda)$. We consider this viewpoint of looking at $S_g$ as it helps to analyse curves on $S_g$ by looking at their arcs on the faces of $S_g$.

Let $\alpha$ be a curve on $S_g$ which intersects $\mu$ and $\lambda$ minimally. Any arc in $\alpha \cap R_\lambda$ with end points on distinct components of $\partial(R_\lambda)$ is called a \textit{strand of $\alpha$} in $R_\lambda$. If $\alpha$ is such that $i(\alpha, \lambda)=1$ and $\alpha \cap \mu \subset R_\lambda$ then $\alpha$ is called a \textit{standard single strand curve}.

Given an ordered set of points on $\mu$, we now give a shorthand notation to represent the arcs of $\mu$ between these points. Let $\mu$ be with a preferred orientation and $x_1, \dots, x_{m \geq 3}$ be distinct points on $\mu$. Considering $\mu$ as the embedding $\mu : [0,1] \longrightarrow S_g$ with $\mu(0) = \mu(1)$, we say that  $x_1, \dots, x_m$ are \textit{along the orientation of $\mu$} if $\mu^{-1}(x_i) < \mu^{-1}(x_{i+1})$ for $i \in \{1, \dots, m-1\}$. We use $\mu_{[x_i, x_{i+1}]}$ to denote the undirected arc of $a$ with end points $x_i$, $x_{i+1}$ and which has no other $x_j$'s on it. Since $\mu_{[x_i, x_{i+1}]}$ is undirected, we set $\mu_{[x_i, x_{i+1}]} = \mu_{[x_{i+1}, x_i]}$. For $i\in \{1, \dots, m\}$, let $b_i$ be curves or essential arcs on $S_g$ such that $b_i \cap \mu = x_i$. When the context is clear, we will interchangeably  use $\mu_{[x_i, x_{i+1}]}$ and $\mu_{[b_i, b_{i+1}]}$.

Let $a_0$, $a_1$, $a_2$, $a_3$, $a_4$ be a geodesic in $\cc$. Let $i(a_0, a_4) = k$ and set $K = \{1, \dots, k\}$. Select some orientation for $a_0$ and $a_4$. Let $a_0 \cap a_4$ = $\{w_i : i \in K \}$ be ordered along the orientation of $a_4$. For $i \in K$, let $a_0^i$ be the arc of $a_0 \cap R_{a_4}$ containing $w_i$. Let the two component curves of $\partial(R_{a_4})$ be $\partial_+(R_{a_4})$ and $\partial_-(R_{a_4})$ such that $a_0^1$ with the induced orientation from $a_0$ goes from $\partial_+(R_{a_4})$ to $\partial_-(R_{a_4})$. There is a natural orientation of $\partial_+(R_{a_4})$ and $\partial_-(R_{a_4})$ induced by the orientation of $a_4$. For $i \in K$, let $u_i = a_0^i \cap \partial_+(R_{a_4})$ and $v_i = a_0^i \cap \partial_-(R_{a_4})$. We call the rectangle in $R_{a_4}$ with boundaries $a_{4_{[w_i, w_{i+1}]}}$, $\partial_+(R_{a_4})_{[u_i, u_{i+1}]}$, $a_{0_{[u_i, w_i]}}$ and $a_{0_{[u_{i+1}, w_{i+1}]}}$ as a \textit{top bucket} and denote it by $T_i$. Similarly, we call the rectangle in $R_{a_4}$ with boundaries $a_{4_{[w_i, w_{i+1}]}}$, $\partial_-(R_{a_4})_{[v_i, v_{i+1}]}$, $a_{0_{[w_i, v_i]}}$ and $a_{0_{[w_{i+1}, v_{i+1}]}}$ as a \textit{bottom bucket} and denote it by $B_i$. Figure \ref{img:bucket} gives a schematic of a top and a bottom bucket. We note that each top and bottom bucket is contained in a unique component of $S_g \setminus (a_0 \cup a_4)$.  Let $H$ be a top (or, bottom) bucket and let $O$ be the component of $S_g \setminus (a_0 \cup a_4)$ containing $H$. We then call $H$ to be a \textit{top (or, bottom) bucket in $O$}.

\begin{figure}
\centering
\includegraphics[scale=0.6]{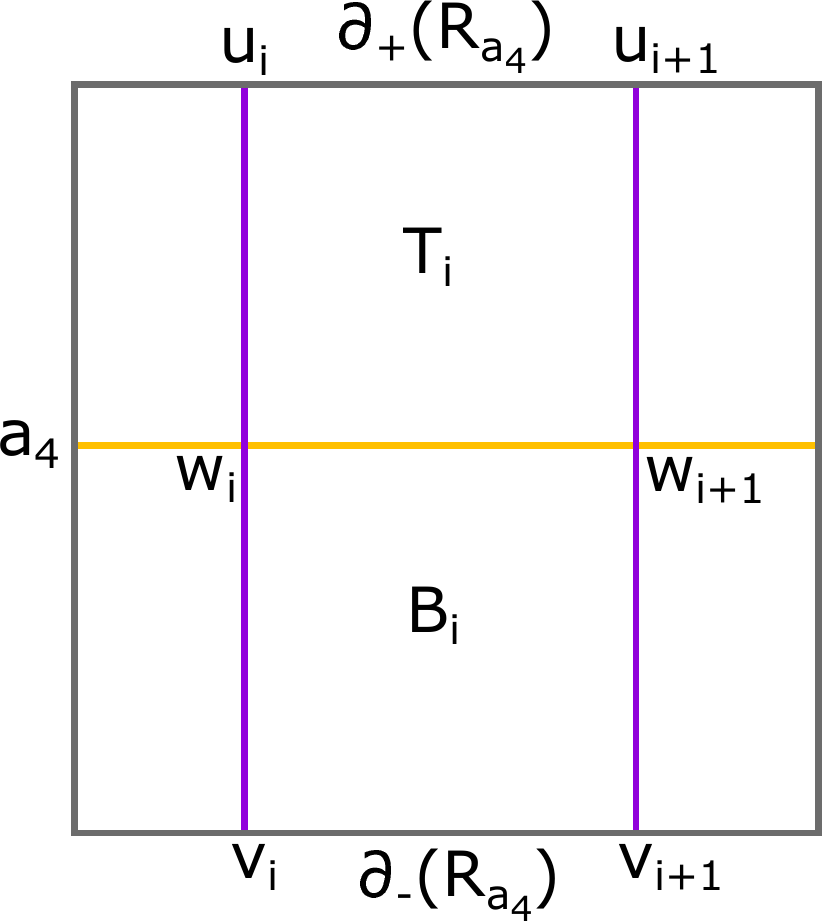}
\caption{Top bucket $T_i$ and bottom bucket $B_i$}
\label{img:bucket}
\end{figure}

Considering $\taf$ to be in special position w.r.t. $a_0$ and $a_4$, there exists a representative of $a_0$ such that the possible schematics of the strands of $\taf$ in $R_{a_4}$ are as in figure \ref{fig:rectified_delta}. The details to the choice of such a representative of $a_0$ can be found in Step $1$ (first paragraph) of the proof of theorem $2$ in \cite{chicken}. Any path between $a_0$ and $\taf$ is of the form $\taf$, $\delta$, $c$, $\Theta$, $a_0$ where $\delta \in B_1(\taf)$, $c \in B_1(\delta)$ and $\Theta$ is a non-trivial path. We now give an algorithm to select a representative of $\delta$ such that for each strand of $\delta$ in $R_{a_4}$ there exists $i \in K$ such that the end points of the strand lies on $\partial_+(R_{a_4})_{[u_i, u_{i+1}]}$ and $\partial_-(R_{a_4})_{[v_i, v_{i+1}]}$. Applying $T_{a_4}$ to a geodesic between $a_0$ and $a_4$, we get that $d(a_4, T_{a_4}(a_0))$ = $4$. Since $\delta \in B_1(\taf)$, we have that $d(\delta, a_4) \geq 3$. From \cite{AH1}, we have that $i_{min}(g, 3) \geq 4$. Thus $i(\delta, a_4) \geq 4$, i.e. there are at least $4$ strands of $\delta$ in $R_{a_4}$. Consider a representative of $\delta$ which is in minimal position with $a_0$, $a_4$, $T_{a_4}(a_0)$, $\partial_+(R_{a_4})$ and $\partial_-(R_{a_4})$. We can choose a representative of $\delta$ such that $\delta \cap a_0 \subset R_{a_4}$ by performing the isotopy $I_2$ described in the step $2$ of Theorem $2$ in \cite{chicken}. An intuitive picture of this isotopy is to finger push the points in $\delta \cap a_0$ which don't lie in $R_{a_4}$, along $a_0$, into $R_{a_4}$. This ``finger pushing'' doesn't disturb the minimal position of $\delta$ and $\taf$. The strands of $\delta$ in $R_{a_4}$ attained after performing the above isotopies can be one of the four possible schematics as in figure \ref{fig:rectified_delta}. If a strand of $\delta$ in $R_{a_4}$, say $\delta'$, is as in figure \ref{fig:rectified_delta_a} or \ref{fig:rectified_delta_b}, we can perform an isotopy of $\delta$ such that the isotopic image of $\delta'$ is as in figure \ref{fig:rectified_delta_c} or \ref{fig:rectified_delta_d} and the isotopy doesn't disturb the other strands of $\delta$. This isotopy of $\delta$ is defined as $I_3$ in the step $3$ of theorem $2$ in \cite{chicken}. The isotopic copy thus obtained is said to be ``$\delta$ in a rectified position" (Step 3, Theorem 2, \cite{chicken}). The isotopies $I_1$, $I_2$ and $I_3$ in \cite{chicken} described for curves $\gamma \in B_1(a_3)$ can also be applied to $\delta$ because they only use the fact that $d(a_0, a_3) \geq 3$.

\begin{figure}
     \centering
     \begin{subfigure}{0.3\textwidth}
         \centering
	\includegraphics[width=\textwidth]{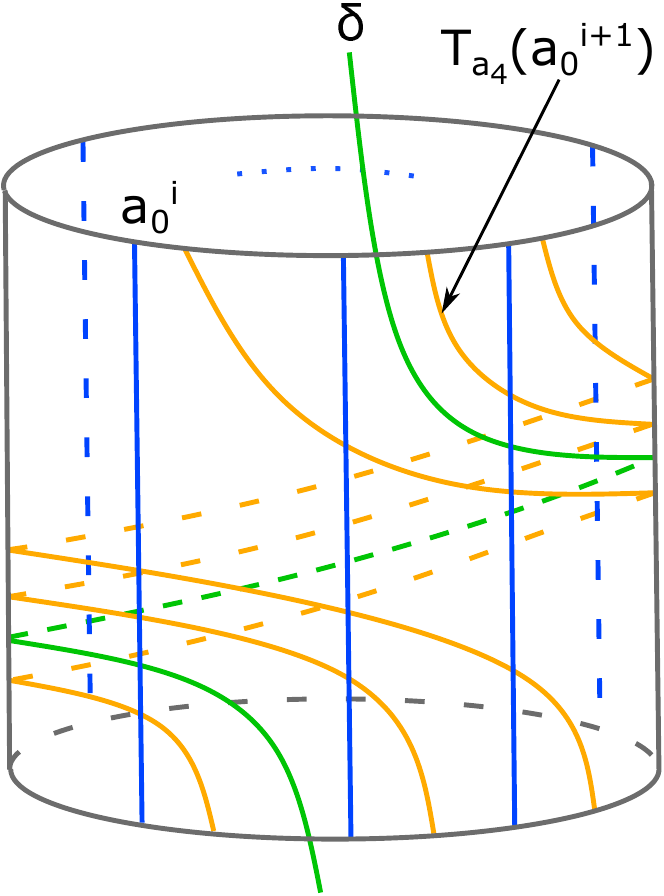}
	\caption{}
	\label{fig:rectified_delta_a}
     \end{subfigure}
     \hspace{0.2\textwidth}
     \begin{subfigure}{0.3\textwidth}
         \centering
	\includegraphics[width=\textwidth]{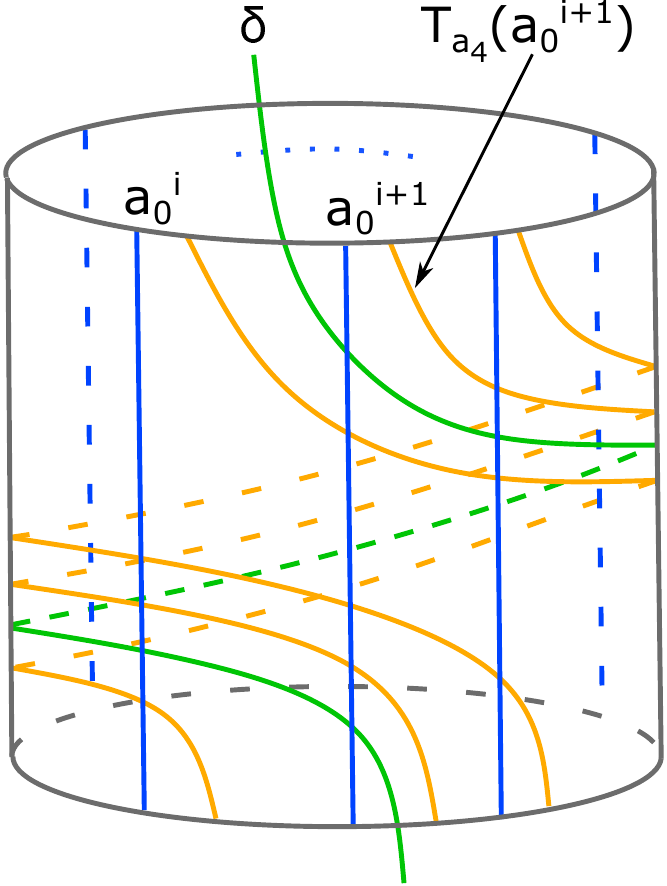}
	\caption{}
	\label{fig:rectified_delta_b}
     \end{subfigure}
     \newline
     \begin{subfigure}{0.3\textwidth}
         \centering
	\includegraphics[width=\textwidth]{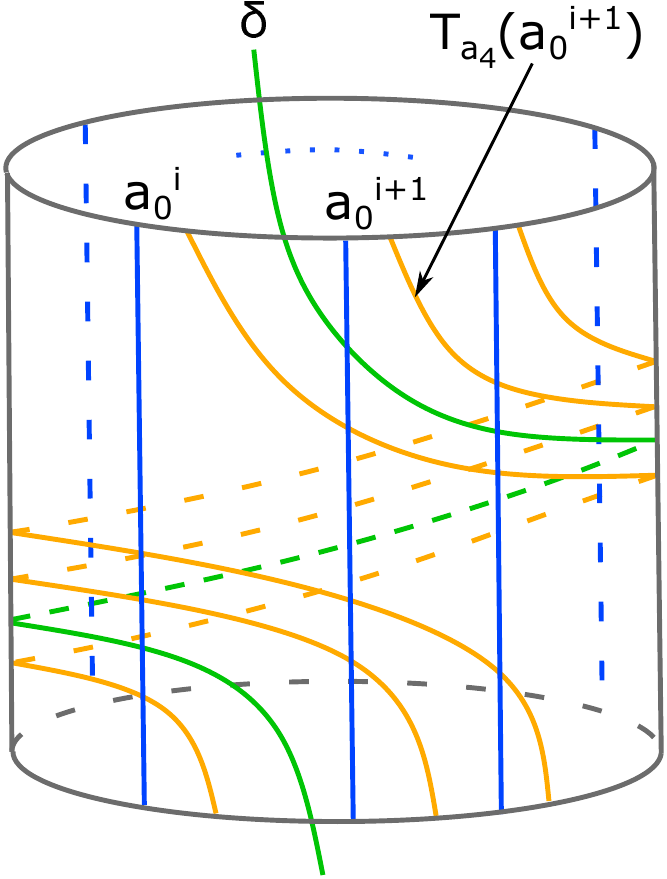}
	\caption{}
	\label{fig:rectified_delta_c}
     \end{subfigure}
      \hspace{0.2\textwidth}
     \begin{subfigure}{0.3\textwidth}
         \centering
	\includegraphics[width=\textwidth]{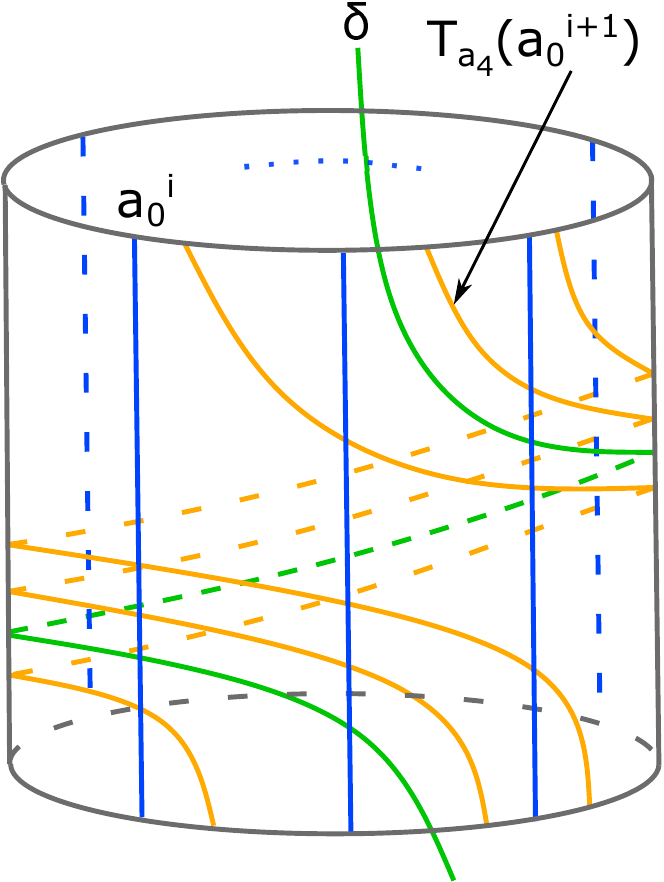}
	\caption{}
	\label{fig:rectified_delta_d}
     \end{subfigure}
     \caption{The possible starting and ending points of strands of $\delta$ in $R_{a_4}$}
     \label{fig:rectified_delta}
\end{figure}

The following proposition \ref{prop:geq4} is a conclusion drawn from the proof of theorem \ref{thm:chckn} in \cite{chicken}.

\begin{thm}[Theorem $2$, \cite{chicken}]
\label{thm:chckn}
Let $S_g$ be a closed surface of genus $g \geq 2$. Let $\alpha$ and $\gamma$ be two curves on $S_g$ with $d(\alpha, \gamma) = 3$. Then, $d(T_\gamma(\alpha), \alpha) = 4$.
\end{thm}

\begin{prop}
\label{prop:geq4}
Let $S_g$ be a closed surface of genus $g \geq 2$. Let $\alpha$ and $\gamma$ be two curves on $S_g$ with $d(\alpha, \gamma) \geq 3$. Then, $d(\alpha, T_\gamma(\alpha)) \geq 4$.
\end{prop}

\begin{proof}
Since $d(\alpha, \gamma) \geq 3$, $\alpha$ and $\gamma$ form a filling pair of curves on $S_g$. Let $\delta \in B_1(T_\gamma(\alpha))$. In the proof of the above theorem \ref{thm:chckn} in \cite{chicken}, it has been shown that if $\alpha$ and $\gamma$ fill $S_g$ then $\delta$ and $\alpha$ fill $S_g$. This implies that $d(\delta, \alpha) \geq 3$. Since $\delta$ is arbitrary, any path from $T_\gamma(\alpha)$ to $\alpha$ will be of length at least $4$. Thus, $d(a, T_\gamma(\alpha)) \geq 4$.
\end{proof}

We have that $a_0$, $a_1$, $a_2$, $a_3 = T_{a_4}(a_3)$, $T_{a_4}(a_2)$, $T_{a_4}(a_1)$, $T_{a_4}(a_0)$ is a path of length $6$ in $\cc$. Existence of a path of length $6$ between $\taf$, $a_0$ and proposition \ref{prop:geq4} gives that $$4 \leq d(a_0, \taf) \leq 6.$$  

Let $\alpha$ and $\beta$ be a filling pair of curves on $S_g$. Let $\beta' \in \beta \setminus \alpha$. Let $D$ be the polygonal disc obtained by gluing the two components of $S_g \setminus (\alpha \cup \beta)$ along $\beta'$. Let $b_1$ be an arc in $D$ such that $b_1$ and $\beta'$ have their end points on the same arcs of $\alpha \cap D$. We say that $b_1$ covers $\beta'$ if $b_1$ is isotopic to $\beta'$ by an isotopy of arcs in $D$ having end points on the same arcs of $\alpha \cap D$ as $\beta'$ and $b_1$.  Let $\mathcal{A}$ be a non-empty set of essential arcs on $S_g \setminus \alpha$ such that the end points of every arc in $\mathcal{A}$ lies on the boundary. We call $\mathcal{A}$ a \textit{filling system} of arcs of $S_g \setminus \alpha$ if the components of $(S_g \setminus \alpha) \setminus \mathcal{A}$ are discs.

\begin{lem}
\label{lem:cover}
Let $\alpha$ and $\beta$ be a pair of filling curves on $S_g$. Let $i(\alpha, \beta) =n$ and the components of $\beta \setminus \alpha$ be $\{\beta_i : 1 \leq i \leq n\}$. Let $\Gamma$ be a non-empty set of essential arcs on $S_g \setminus \alpha$. If for every $i \in \{1, \dots, n\}$, there exists $g_i \in \Gamma$ such that $g_i$ covers $\beta_i$, then $\Gamma$ is a filling system of arcs of $S_g \setminus \alpha$.
\end{lem}
\begin{proof}
Consider the components of $(S_g \setminus \alpha) \setminus \cup \{g_i\}_{1 \leq i \leq n}$. These components coincide with the components of $(S_g \setminus \alpha) \setminus \beta$ and hence, are discs. Since $\cup_{1 \leq i \leq n} \{g_i\} \subset \Gamma$, the components of $(S_g \setminus \alpha) \setminus \Gamma$ are also discs.
\end{proof}

Let $D$ be a component of $S_g \setminus (a_0 \cup a_4)$. Let $T_p$, $T_q$ be top buckets in $R_{a_4}$ for some $p, q \in K$, $p < q$ such that $(T_p \cup T_q) \subset D$. Let $\gamma''$ be an arc in $\cup_{i = p+1}^{q-1}T_i$ parallel to $a_4$ with end points on $a_0^{p+1} \cap T_{p+1}$ and $a_0^q \cap T_{q-1}$. Let $\gamma'$ be an arc in the interior of $D$ with end points $(\gamma'' \cap a_0) \cap D$. Let $\gamma$ be the curve obtained by concatenation of the arcs $\gamma'$ and $\gamma''$. A schematic of $\gamma$ is shown in figure \ref{fig:gamma_fill}. We call $\gamma$ a \textit{scaling curve from $T_p$ to $T_q$}. Since $R_{a_4}$ is a cylinder, we can similarly define a \textit{scaling curve from $T_q$ to $T_p$} as follows. Let $\gamma_1''$ be an arc in $\cup_{i = q+1}^{p-1}T_i$ parallel to $a_4$ with end points on $a_0^{q+1} \cap T_{q+1}$ and $a_0^p \cap T_{p-1}$. Let $\gamma_1'$ be an arc in the interior of $D$ with end points $(\gamma_1'' \cap a_0) \cap D$. Then the curve obtained by concatenation of the arcs $\gamma'$ and $\gamma''$ is a scaling curve from $T_q$ to $T_p$. By replacing top buckets with their bottom buckets counterpart, we can define \textit{scaling curve from $B_p$ to $B_q$} and \textit{$B_q$ to $B_p$}. 

\begin{figure}
\centering
\includegraphics[scale=0.33]{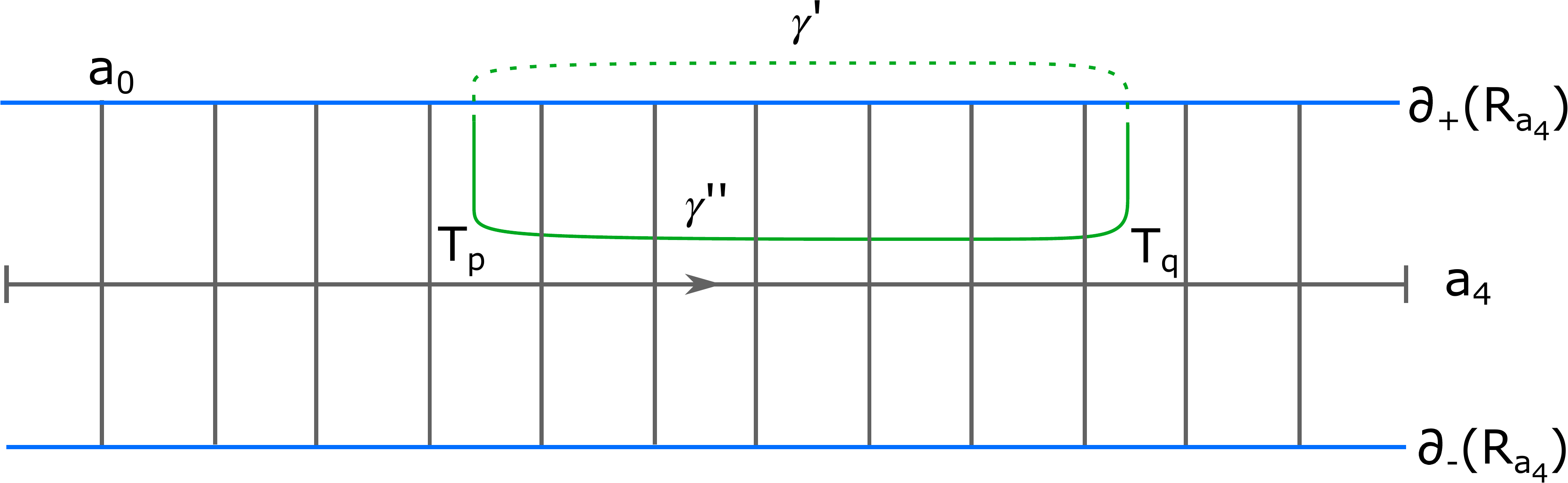}
\caption{A schematic of the scaling curve $\gamma$ from $T_p$ to $T_q$. The dashed arc is a schematic of $\gamma'$.}
\label{fig:gamma_fill}
\end{figure}

\begin{lem}
\label{lem:gamma_nonnull}
Scaling curves are not null-homotopic. 
\end{lem}

\begin{proof}
We prove the lemma when $\gamma$ is a scaling curve from a top bucket $T_p$ to $T_q$, $p < q$ and $(T_p \cup T_q) \subset D$ for some component $D$ of $S_g \setminus (a_0 \cup a_4)$. A similar proof follows if $\gamma$ is a scaling curve from a bottom bucket $B_p$ to $B_q$, $p < q$ by replacing $T_p$, $T_q$ with $B_p$, $B_q$, respectively, in the proof below. Similar proofs work for scaling curves from $T_q$ to $T_p$ and $B_q$ to $B_p$. We show that $\gamma$ is not null-homotopic by considering a minimal representative of $\gamma$ along with $a_0$ and showing that this representative has non-zero intersections with $a_0$. We obtain this minimal representative of $\gamma$ and $a_0$ by removing bigons in iterations.

Suppose if possible that $\gamma$ and $a_0$ are not in minimal position. Since there exists an isotopic copy of $\gamma$ such that $\gamma''$ overlaps with $a_{4_{[w_{p+1}, w_{q-1}]}}$, if $\gamma$ and $a_0$ are not in minimal position then a bigon is formed by $\gamma'$ and a subarc of $a_0$. This subarc of $a_0$ is a component of $a_0 \setminus a_4$ because otherwise, if there is a point of $a_0 \cap a_4$ on the boundary of this bigon then as $\gamma \cap a_4 = \phi$ we get a bigon between $a_0$ and $a_4$ which contradicts the minimality of $a_0$, $a_4$. The closed component of $a_0 \setminus a_4$ that contains this subarc also contains the arcs $T_p \cap a_{0_{[u_{p+1}, w_{p+1}]}}$ and $T_q \cap a_{0_{[u_{q-1}, w_{q-1}]}}$. Thus, $(T_{p+1} \cup T_{q-1}) \subset D_1$ for some component $D_1$ of $S_g \setminus (a_0 \cup a_4)$. We remove this bigon between $\gamma$ and $a_0$ to obtain an isotopic copy of $\gamma$. This isotopic copy of $\gamma$ is in turn a scaling curve from $T_{p+1}$ to $T_{q-1}$. By abuse of notation, we denote this isotopic copy as $\gamma$. 

If we have that $\gamma$ is not in minimal position with $a_0$, then by similar arguments as in the previous paragraph, $T_{p+1} \cap a_{0_{[u_{p+2}, w_{p+2}]}}$ and $T_{q-1} \cap a_{0_{[u_{q-2}, w_{q-2}]}}$ are contained in the same closed component of $a_0 \setminus a_4$. Thus, we have that $D_1$ is a rectangle. As previously, we remove this bigon between $\gamma$ and $a_0$ and consider denote the new isotopic copy which is also a scaling curve from $T_{p+2}$ to $T_{q-2}$ by $\gamma$. Further, we have that $(T_{p+2} \cup T_{q-2}) \subset D_2$ for some component $D_2$ of $S_g \setminus (a_0 \cup a_4)$.

Continuing in a similar iterative manner as in the above paragraphs, if $\gamma$ and $a_0$ are not in minimal intersection position, then we claim that there is a positive integer $l$ with $l < \lceil \frac{q-p}{2} \rceil -1$ such that 
\begin{enumerate}
	\item $T_{p+l} \cup T_{q-l}$ is contained in the one component $D_l$ of $S_g \setminus (a_0 \cup a_4)$
	\item $T_{p+l} \cap a_4$ and $T_{q-l} \cap a_4$ are separated by a single edge corresponding to $a_0$ in $D_l$
\end{enumerate}
The fact that there exists $l$ with $l \leq \lceil \frac{q-p}{2} \rceil -1$ is immediate as there are at most $\lceil \frac{q-p}{2} \rceil$ pairs of buckets of the form $T_{p+i}$ and $T_{q-i}$ between $T_p$ and $T_q$ in $R_{a_4}$. We first show that if we assume $l= \lceil \frac{q-p}{2} \rceil -1$ along with hypothesis $(1)$ and $(2)$ then we arrive at the following contradictions. If $q-p$ is odd then we have that $T_{p+l}$ and $T_{q-l}$ are adjacent top buckets. But if $T_{p+l}$ and $T_{q-l}$ are adjacent top buckets then $a_0$ has a self intersection, which is absurd. If $q-p$ is even then $p+l+2 = q-l$. But then the $a_0 \cap D_l$ arc containing the end points $w_{p+l+1}$ and $w_{q-l}$ encloses a disc with $a_{4_{[w_{p+l+1}, w_{q-l}]}}$, thus giving a bigon between $a_0$ and $a_4$. This contradicts that $a_0$ and $a_4$ are in minimal position.

We thus have that the scaling curve from $T_{p+l}$ to $T_{q-l}$ intersects $a_0$ minimally and is isotopic to the given $\gamma$.

\end{proof}

As in the proof of lemma \ref{lem:gamma_nonnull}, whenever we consider a scaling curve we will work with an isotopic copy of it which is in minimal position with $a_0$, $a_4$ and $\taf$.

\begin{cor}
\label{lem:gamma_fill}
Scaling curves fill with $a_0$.   
\end{cor}

\begin{proof}
From the construction of a scaling curve, $\gamma$, $\gamma \cap a_4 = \phi$ and hence $d(\gamma, a_0) \geq 3$. Thus, $a_0$ and $\gamma$ fill $S_g$.
\end{proof}

\begin{rem}
If $a_0$ and $a_4$ intersect $i_{min}(g, 4)$ number of times then any scaling curve are at a distance $3$ from $a_0$.
\end{rem}

The following lemma \ref{lem:self_gluing} states that for the purpose of cutting $S_g \setminus a_0$ into discs, not every arc of $a_4 \setminus a_0$ is necessary. We can forgo any one of the arcs of $a_4 \setminus a_0$.

\begin{lem}
\label{lem:self_gluing}
Let $b$ be a component of $a_4 \setminus a_0$. Then $(a_4 \setminus a_0) \setminus b$ is a filling system of arcs of $S_g \setminus a_0$.
\end{lem}
\begin{proof} Each component of $a_4 \setminus a_0$ is common to two components of $S_g \setminus (a_0 \cup a_4)$. Let the components of $S_g \setminus (a_0 \cup a_4)$ that share the edge corresponding to $b$ be $D_1$ and $D_2$. 

We first show that $D_1 \neq D_2$. On the contrary, if $D_1 = D_2$ let $p$ be the central curve of the annulus obtained by gluing $D_1$ along $b$. Being in minimal position with $a_0$ and $a_4$, $p$ forms an essential curve on $S_g$. Since $i(p, a_0) = 0$ and $i(p, a_4) =1$, we get a path of distance $3$ between $a_0$ and $a_4$ via $p$, which is not possible. 

Call the disc obtained by gluing $D_1$ and $D_2$ along $b$ as $D$. Components of $S_g \setminus ((a_4 \setminus a_0) \setminus b)$ comprise of the components of $(S_g \setminus (a_0 \cup a_4)) \setminus (D_1 \cup D_2)$ and $D$. Since each component is a disc, it follows that $(a_4 \setminus a_0) \setminus b$ forms a filling system of $S_g \setminus a_0$.
\end{proof}

The following lemmas explain a few observations regarding the buckets in $R_{a_4}$ and the components of $S_g \setminus (a_0 \cup a_4)$ that contain them. 

From corollary \ref{lem:gamma_fill} and the fact $i_{min}(g, 3) \geq 4$ (\cite{AH1}), we have the following corollary regarding the placement in $R_{a_4}$ of the top buckets which are subset of the same disc of $S_g \setminus (a_0 \cup a_4)$. A similar version of corollary \ref{cor:separation} holds true for bottom buckets.

\begin{cor}
\label{cor:separation}
Let $D$ be a component of $S_g \setminus (a_0 \cup a_4)$ and $T_p$, $T_q$ be distinct top buckets in $R_{a_4}$ for some $p, q \in K$, $p < q$ such that $(T_p \cup T_q) \subset D$. Then $|q-p| \geq 4$.
\end{cor}

\begin{lem}
\label{lem:bucket_stack}
For any $i \in K$, $T_i$ and $B_i$ can't be subsets of the same component of $S_g \setminus (a_0 \cup a_4)$.
\end{lem}

\begin{proof}
This follows directly from the proof of lemma \ref{lem:self_gluing} by considering $b = a_{4_{[i, i+1]}}$.
\end{proof}

Let $\delta_1$ be the point on $\partial_+(R_{a_4})$ such that $\delta_1 = \delta \cap \partial_+(R_{a_4})_{[u_1, u_2]}$ and that one of the arcs $\partial_+(R_{a_4}) \setminus \{u_1\cup \delta_1\}$ doesn't contain any points of $\delta \cap \partial_+(R_{a_4})$.  If $m = i(a_4, \delta)$, let $\delta \cap \partial_+(R_{a_4}) = \{\delta_i : 1 \leq i \leq m\}$ such that the $\delta_1$, $\delta_2$, \dots, $\delta_m$ are along the orientation of $\partial_+(R_{a_4})$. Let the strand of $\delta$ containing the point $\delta_i$ be $\delta^i$.

Let $\delta^r$, $\delta^s$ be any two distinct strands of $\delta$ in $R_{a_4}$ such that $\delta^r$ and $\delta^s$ start in distinct top buckets, say $T_r$ and $T_s$, and that there exists a component of $\partial_+(R_{a_4}) \setminus (\delta^r \cup \delta^s)$ that doesn't contain any points of $\partial_+(R_{a_4}) \cap \delta$ other than $\partial_+(R_{a_4}) \cap \delta^r$ and $\partial_+(R_{a_4}) \cap \delta^s$. We call the rectangular component of $R_{a_4} \setminus (\delta^r \cup \delta^s)$ which doesn't contain any other strand of $\delta$ as a \textit{$\delta$-track in $R_{a_4}$} and denote it by $\delta ^{r,s}$. The boundary of $\delta^{r,s}$ comprises of the arcs $\delta^r$, $\delta^s$, $\partial_+(R_{a_4})_{[\delta^r, \delta^s]}$ and $\partial_-(R_{a_4})_{[\delta^r, \delta^s]}$. Further, assuming $r < s$, we call the set $\bigcup_{i=r}^{s}(T_i \cup B_i)$ as \textit{inside of $\delta^{r,s}$}.

\begin{lem}
\label{lem:bucket_lemma}
Let $O$ be a $2n$-gon disc of $S_g \setminus (a_0 \cup a_4)$ with $n \geq 4$. For any delta-track in $R_{a_4}$, $\delta^{r,s}$, there exists at-least one top bucket or, bottom bucket in $O$ that is not in the inside of $\delta^{r,s}$.
\end{lem}

\begin{proof}
Let us suppose on the contrary that there exists a $\delta$-track $\delta^{r,s}$ such that all the top and bottom buckets in $O$ are inside $\delta^{r,s}$. Without loss of generality, assume the top and bottom buckets containing the end points of $\delta^r$ are $T_r$ and $B_r$, respectively. Similarly, for $\delta^s$ the top and bottom buckets are $T_s$ and $B_s$, respectively. If every bucket in $O$ are of the form $T_i$ or, $B_i$ for $r < i < s$ then, we get a scaling curve $\gamma$ such that $\delta$, $\gamma$, $a_4$ is a path. This contradicts $d(\delta, a_4) \geq 4$. 

Also, since every bucket of $O$ is inside $\delta^{r,s}$, any arc of $\delta \cap O$ are either an arc that covers $a_{4_{[i, i+1]}}$ for $r < i < s$ or, an arc, say $\delta'$, with end points $\delta^r \cap a_4$ and $\delta^s \cap a_4$. It follows along with lemma \ref{lem:bucket_stack} that either $T^r$, $T^s$ or, $T^r$, $B^s$ or, $B^r$, $T^s$ or, $B^r$, $B^s$ are buckets of $O$. We show that all these possibilities, if they exist, lead to a contradiction. The following is a combinatorial proof and we give it for the case $O$ is an octagon. As $n$ increases, the proof remains intact with only the possibility of certain cases being redundant. 

If $T^r$, $T^s$ are buckets in $O$, figure \ref{fig:bucket_lemma_1} shows the distinct possible $\delta'$. If $T^r$, $B^s$ are buckets in $O$, figure \ref{fig:bucket_lemma_2} shows the distinct possible $\delta'$. If $B^r$, $B^s$ are buckets in $O$, figure \ref{fig:bucket_lemma_3} shows the distinct possible $\delta'$. In the possible cases of figure \ref{fig:bucket_lemma_1}, \ref{fig:bucket_lemma_2} and \ref{fig:bucket_lemma_3a}, by the pigeon hole principle, either component of $O \setminus \delta'$ contains either two top or, bottom buckets. Thus, if these cases occur, we can construct a scaling curve $\gamma$ such that $\delta$, $\gamma$, $a_4$ is a path, which is absurd. For figure \ref{fig:bucket_lemma_3b}, if both the components of $O \setminus \delta'$ contains a top bucket, because $S_g$ is an orientable surface, we will be able to find a non-trivial curve $\gamma$ with properties as in the above cases. Here, $\gamma \cap O$ lies in the component of $O$ containing the vertices $w_{r+1}$ and $w_{s}$.

If $B^r$, $T^s$ are buckets in $O$, the arguments are similar to the case of $T^r$ and $B^s$ being buckets of $O$.

\begin{figure}
     \centering
     \begin{subfigure}{0.35\textwidth}
         \centering
	\includegraphics[width=\textwidth]{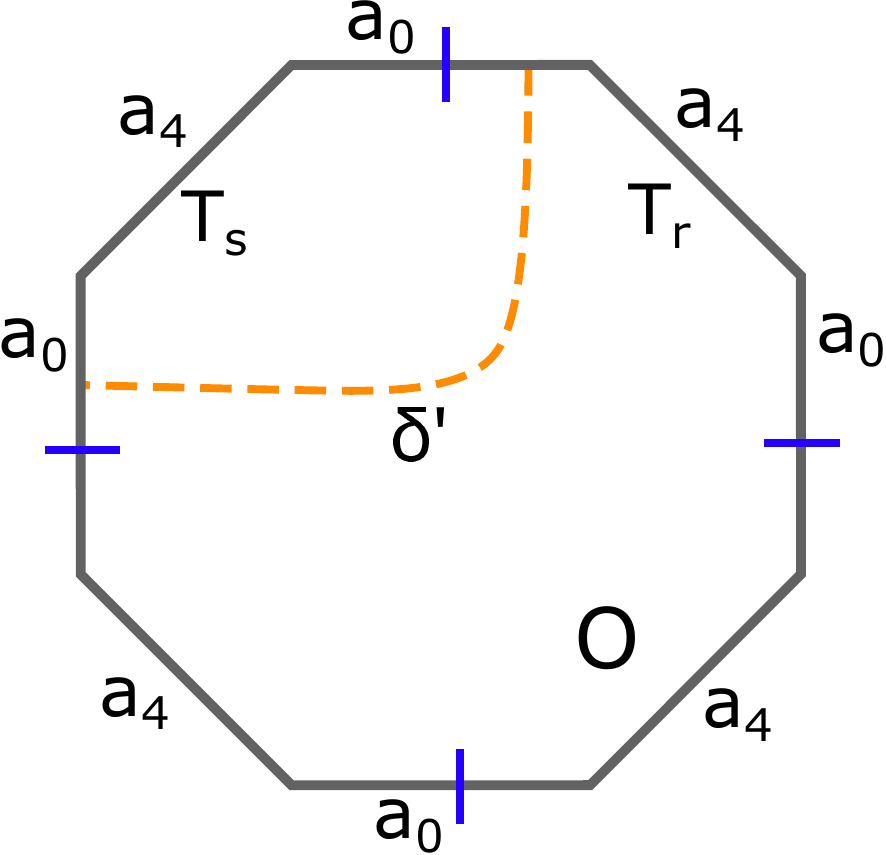}
	\caption{}
	\label{fig:bucket_lemma_1a}
     \end{subfigure}
     \hspace{1cm}
     \begin{subfigure}{0.35\textwidth}
         \centering
	\includegraphics[width=\textwidth]{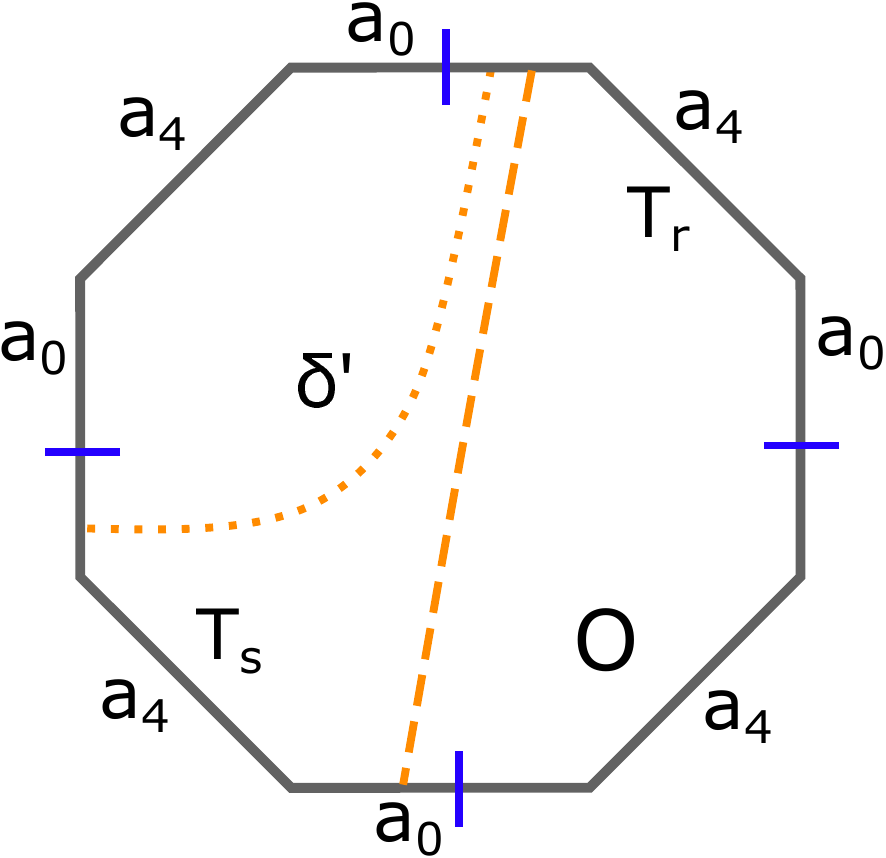}
	\caption{}
	\label{fig:bucket_lemma_1b}
     \end{subfigure}
     \caption{Both the dotted line and the dashed line are possibilities for $\delta'$ if $T_r$ and $T_s$ are as in the schematic.}
     \label{fig:bucket_lemma_1}
\end{figure}

\begin{figure}
     \centering
     \begin{subfigure}{0.35\textwidth}
         \centering
	\includegraphics[width=\textwidth]{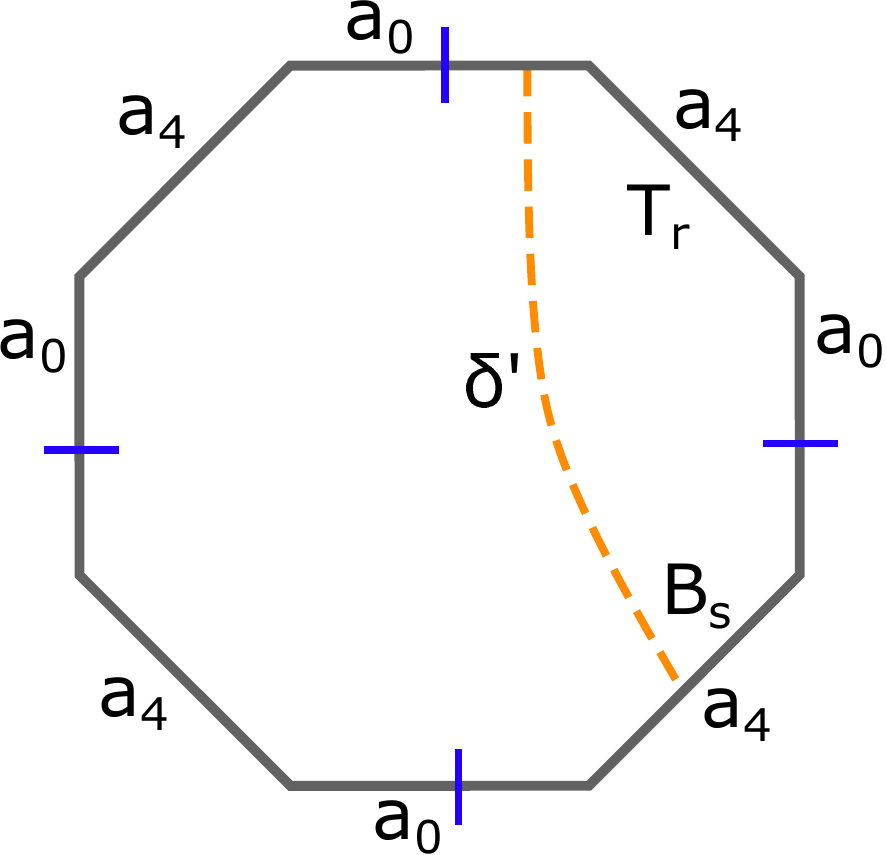}
	\caption{}
	\label{fig:bucket_lemma_2a}
     \end{subfigure}
     \hspace{1cm}
     \begin{subfigure}{0.35\textwidth}
         \centering
	\includegraphics[width=\textwidth]{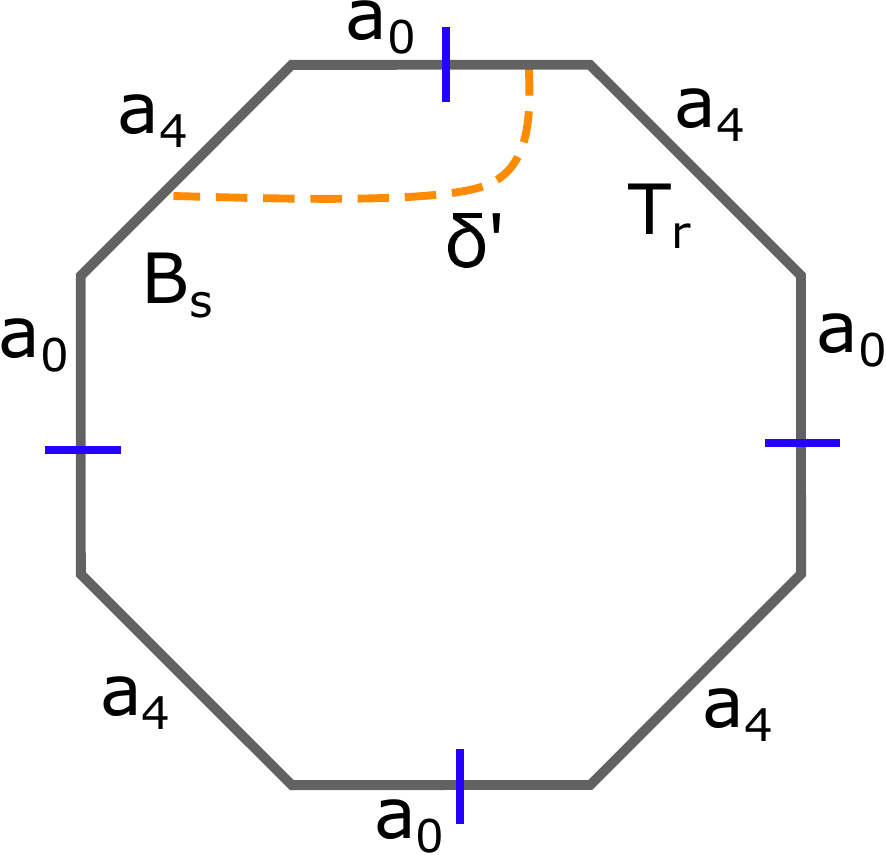}
	\caption{}
	\label{fig:bucket_lemma_2b}
     \end{subfigure}
     \caption{}
     \label{fig:bucket_lemma_2}
\end{figure}

\begin{figure}
     \centering
     \begin{subfigure}{0.35\textwidth}
         \centering
	\includegraphics[width=\textwidth]{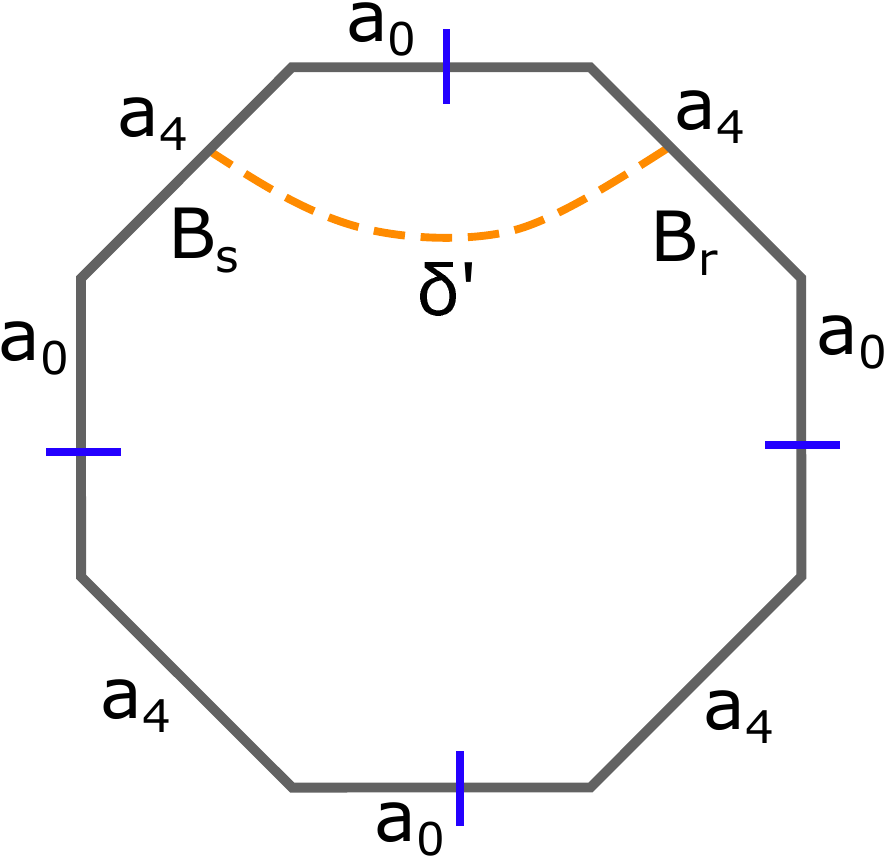}
	\caption{}
	\label{fig:bucket_lemma_3a}
     \end{subfigure}
     \hspace{1cm}
     \begin{subfigure}{0.35\textwidth}
         \centering
	\includegraphics[width=\textwidth]{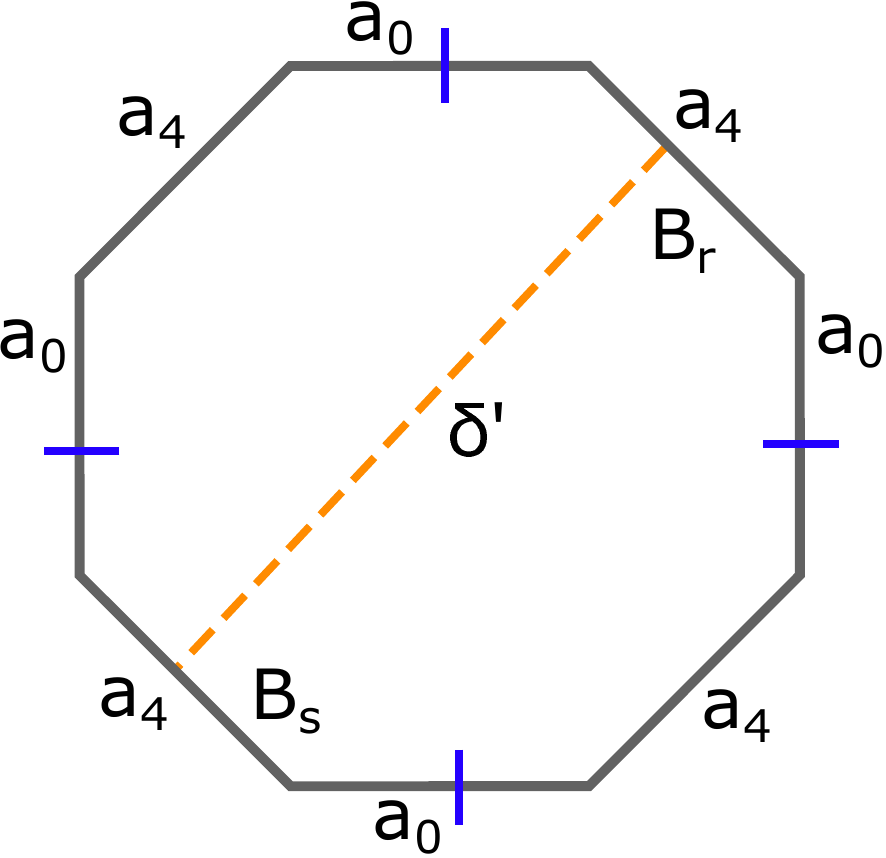}
	\caption{}
	\label{fig:bucket_lemma_3b}
     \end{subfigure}
     \caption{}
     \label{fig:bucket_lemma_3}
\end{figure}

\end{proof}

Suppose we have a filling system of arcs of $S_g \setminus a_0$ and there is another set of arcs on $S_g \setminus a_0$ that covers all the arcs in the former filling system of arcs except for one. In the following we explore a sufficient condition on the latter system of arcs which ensures that it forms a filling system of arcs. To prove this condition we take the aid of the fact that $a_0$ and $a_4$ fill $S_g$.

Let $\mathcal{I}$ be some non-rectangular component of $S_g \setminus (a_0 \cup a_4)$ and $a_0^{J_1}$, $a_0^{J_2}$ be two distinct edges of $\mathcal{I}$, where $J_1, J_2 \in K$.  Consider an arc, $b$, in $\mathcal{I}$ with end points on $a_0^{J_1}$ and $a_0^{J_2}$. Let $I$ be one of the two components of $\mathcal{I} \setminus J$ such that $I$ contains an edge $a_0^J$ for some $J \in K$. By our assumption, there exists an edge, $a_0^{J_3}$, in $I$ such that $a_0^{J_2} \cap I$ and $a_0^{J_3}$ are adjacent in the polygon $I$. Consider arcs, $z$ and $z'$, in $I$ such that the end points of $z$ are $a_0^{J_1} \cap I$, $a_0^{J_3}$ and the endpoints of $z'$ are $a_0^{J_2}$, $a_0^{J_3}$. Clearly, $z'$ covers the $a_4$ edge in $I$ adjacent to $a_0^{J_2} \cap I$ and $a_0^{J_3}$. We call such a pair of arcs $(z, z')$ to \textit{almost cover $b$}. Figure \ref{fig:almost_filling} gives a schematic of $(z,z')$.

\begin{figure}
\centering
\includegraphics[scale=0.5]{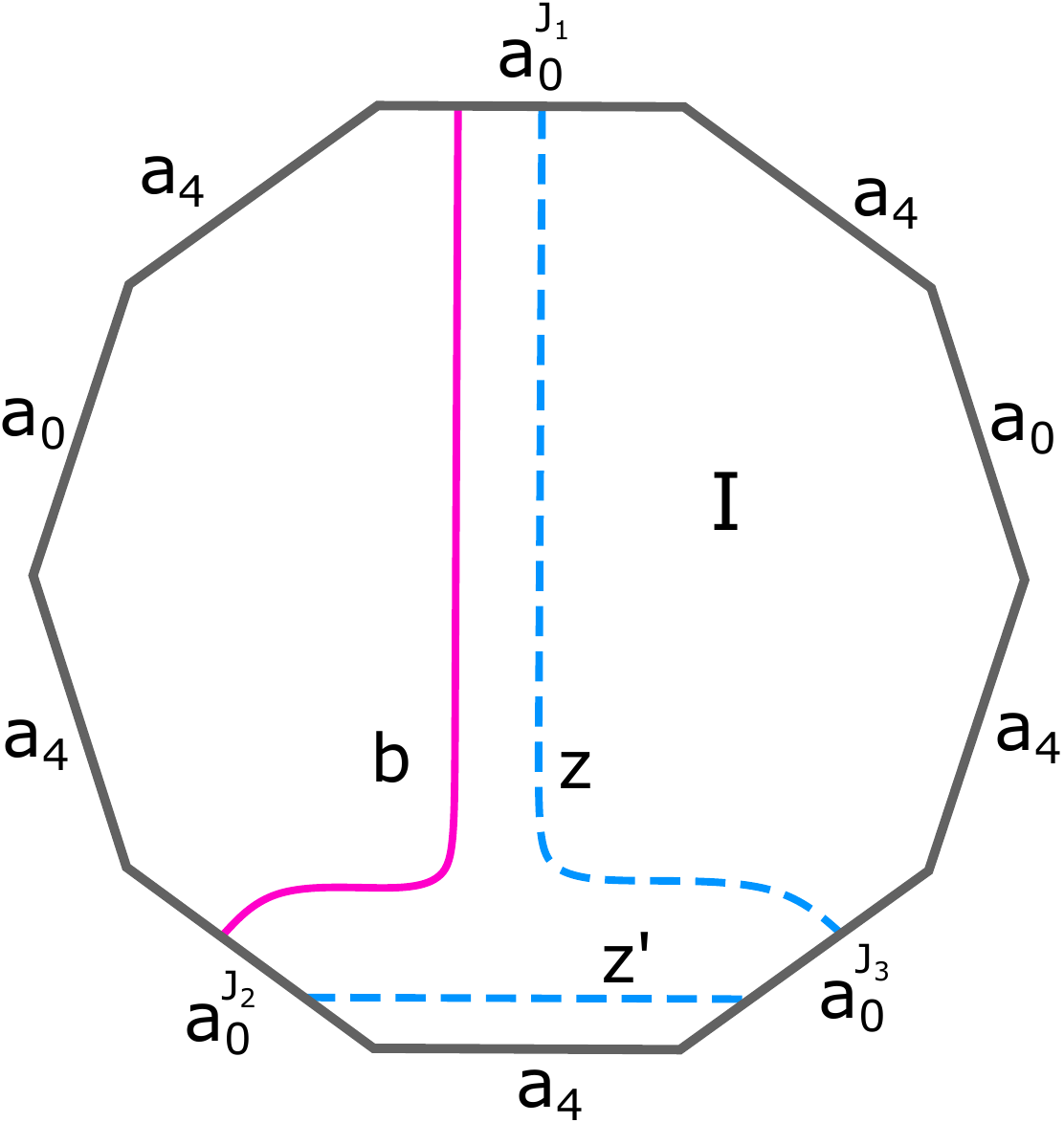}
\caption{A schematic of $\mathcal{I}$ with $(z, z')$ almost covering $b$.}
\label{fig:almost_filling}
\end{figure}

\begin{lem}\label{lem:almost_fill}
Let $a_0$ and $a_4$ be curves on $S_g$ with $d(a_0, a_4)=4$. Let $\kappa$ be another curve on $S_g$ such that $a_0$ and $\kappa$ fill $S_g$. Let $\Gamma$ be a set of essential arcs on $S_g \setminus a_0$. If $\Gamma$ consists of arcs that covers all but one arc of $\kappa \setminus a_0$ and almost covers the remaining arc of $\kappa \setminus a_0$, then $\Gamma$ forms a filling system of arcs of $S_g \setminus a_0$.
\end{lem}

\begin{proof}
Let  $g, g' \in \Gamma$ and $x$ be the component of $\kappa \setminus a_0$ such that $(g, g')$ almost cover $x$. Let $\{g_j\}_{j \in J} \subset \Gamma$ be the arcs that cover the components of $(\kappa \setminus a_0) \setminus x$.  Let $\mathcal{I}$ be the component of $S_g \setminus (a_0 \cup a_4)$ that contains $x$.

Starting with the components of $(S_g \setminus a_0) \setminus a_4$, we can obtain the components of $(S_g \setminus a_0) \setminus \kappa$ by gluing the components of $(S_g \setminus a_0) \setminus a_4$ along the components of $a_4 \setminus a_0$ and cutting along the components of $\kappa \setminus a_0$. In the components $(S_g \setminus (a_0 \cup a_4)) \setminus \mathcal{I}$, the action of cutting along the components of $\kappa \setminus a_0$ coincides with the action of cutting along $\{g_j\}_{j \in J}$.

Let $I$ and $I'$ be the two components of $\mathcal{I} \setminus x$ such that $(g \cup g') \subset I$. In $\mathcal{I}$, the action of gluing along $x$ is such that it separates the $a_4$-edges of $\mathcal{I}$ into two sets i.e. $(a_4 \setminus a_0) \cap I$ and $(a_4 \setminus a_0) \cap I'$. Let $G_1$, $G_2$ and $G_3$ be the components of $\mathcal{I} \setminus (g \cup g')$. From the schematic in figure \ref{fig:almost_filling}, we can see that the components of $\mathcal{I} \setminus (g \cup g')$ can be named such that $I' \subset G_1$, $G_2 \subset I$ and $G_3 \subset I$. Since, the components of $(S_g \setminus a_0) \setminus (\kappa \setminus x)$ and the components of $(S_g \setminus a_0) \setminus \cup_{j \in J} g_j$ are the same, we have that the components of $(S_g \setminus (a_0 \setminus \cup_{j \in J} g_j) \setminus (g \cup g')$ will be discs if the action of cutting along $g \cup g'$ doesn't put an arc from $(a_4 \setminus a_0) \cap I$ and another from $(a_4 \setminus a_0) \cap I'$ in the same $G_i$. Such a phenomenon never occurs by our definition of almost filling.

\end{proof}

\section{Distance $\geq 5$ criterion for $d(a_0, \taf)$}
\label{sec:main}

Geodesics between $a_0$ and $\taf$ in $\cc$ will be of the form $z_0=\taf, z_1 =\delta, z_2= c, \dots, z_N = a_0$ for some $N \in \{4, 5, 6\}$ and $c \in B_1(\delta)$. We have that $d(a_0, \taf) \geq 5$ if and only if $d(a_0, c) \geq 3$ for all possible $c \in B_2(\taf) \cap B_1(\delta)$. In the following we identify the characteristics of $\delta$ and $c$ which results in $d(c, a_0) \geq 3$. These results are compiled in section \ref{sec:conc}.

Since $c \in B_1(\delta)$, $d(\delta, a_4) \geq 3$ implies that $c \cap a_4 \neq \phi$. Thus,  there is at least one strand of $c$ in $R_{a_4}$. We now perform an isotopy of $c$ so that $c$ is in a favourable position with respect to $a_0$, $a_4$, $\partial(R_{a_4})$ on $S_g$. The idea behind this isotopy is to get a representative of $c$ such that in $R_{a_4}$, the strands of $c$ resemble the ``spiral pattern'' of the strands of $\taf$ and $\delta$. Consider an isotopic copy of $c$ such that $c$ is in minimal position with $\partial_+(R_{a_4})$, $\partial_-(R_{a_4})$, $a_0$, $a_4$ and $\delta$. Consider any strand, $c'$, of $c$ in $R_{a_4}$. Let $c_+ = c' \cap \partial_+(R_{a_4})$ and $c_- = c' \cap \partial_-(R_{a_4})$. Let $i_0 \in K = \{1, \dots, k\}$ such that $c_+$ lies on the boundary of the top bucket $T_{i_0}$. There exists $0\leq l \leq k-1$ such that $c_-$ lies on the boundary of the bottom bucket $B_{i_0+l}$. Note that since in any top bucket $T_{i}$, there is an arc of delta with end points on $a_0^i$ and $a_0^{i+1}$, if $l=0$ then $c' \cap a_0^{i_0+1} \neq \phi$. There exists $i_1, i_2 \in K$ such that $c_+$ lies in $\partial_+(R_{a_4})_{[\delta^{i_1}, \delta^{i_2}]}$. Since $c' \cap (\delta^{i_1} \cup \delta^{i_2}) = \phi$, we have $i_1 \leq i_0+l \leq i_2$. Consider the arcs $c_1$, $c_2$ and $c_3$ in the annulus $R_{a_4}$ as follows :
\begin{itemize} 
\item $c_1$ starts at $c_+$ passing through $T_{i_0}$, $T_{i_0+1}$ \dots, $T_{i_1-1}$ and ends in some interior point on $a_0^{i_1} \cap T_{i_1-1}$
\item $c_2$ is an arc parallel to $a_4$ which starts at $c_1 \cap a_0^{i_1}$, passes through $T_{i_1}$, $T_{i_1+1}$ \dots, $T_{i_0+l-1}$ and ends in some interior point on $a_0^{i_0+l} \cap T_{i_0+l}$
\item $c_3$ is an arc in the rectangle $T_{i_0+l} \cup B_{i_0+l}$ with end points $c_2 \cap a_0^{i_0+l}$ and $c_-$
\end{itemize}

Let $c''$ be the arc obtained by concatenating $c_1$, $c_2$ and $c_3$. Note that $c''$ intersects $a_4$ only once. Since both $c'$ and $c''$ are arcs in the rectangle with edges $\delta^{i_1}$, $\delta^{i_2}$, $\partial_+(R_{a_4})_{[\delta^{i_1}, \delta^{i_2}]}$ and $\partial_-(R_{a_4})_{[\delta^{i_1}, \delta^{i_2}]}$ such that both $c'$ and $c''$ have end points $c_+$ and $c_-$, there is a end point fixing isotopy, $\mathcal{I}$, of arcs in the rectangle from $c'$ to $c''$. The isotopy $\mathcal{I}$ can be extended to an isotopy of $c$ to $(c \setminus c') \cup c''$ such that the action on $c \setminus c'$ remains identity. By abuse of notation, we denote $\mathcal{I}(c')$ i.e. $c''$ by $c'$. Since the strand of $c$, $c'$, is arbitrary and $\mathcal{I}$ is identity on $c \setminus c'$, we can apply $\mathcal{I}$ to every strand of $c$ to obtain a representative of $c$ which remains in minimal position with $a_0$, $a_4$, $\taf$ and $\delta$. We will always consider such a representative of $c$.

Suppose $i(c, a_4) = k_0$. Let $c_1$, \dots, $c_{k_0}$ be the strands of $c$. Let the end points of $c_i$ be in the top bucket $T_{d_i}$ and the bottom bucket $B_{d_i-l_i}$. We call the set $C_i = \bigcup_{j=d_i-l_i}^{d}(T_j \cup B_j)$ as the inside of $c_i$. We define $\bigcap_{j=1}^{k_0}(C_j)$ as the \textit{inside of $c$}.

We now look into the possible values for $d(c, a_0)$ by considering the following two cases depending on the number of strands of $c$ :
\begin{enumerate}
\item[Case i :] there is a single strand of $c$ in $R_{a_4}$
\item[Case ii :] there are multiple strands of $c$ in $R_{a_4}$ 
\end{enumerate}

\textbf{Case i :} Suppose there exists a single strand, $c'$, of $c$ in $R_{a_4}$. Let $c' \cap \partial_+(R_{a_4})$ lie on $\partial_+(R_{a_4})_{[\delta^p, \delta^q]}$ where $1\leq p< q \leq m$. We first consider the case when $\delta^p \cap \partial_+(R_{a_4})$ and $\delta^q \cap \partial_+(R_{a_4})$ lie in the same top bucket then $c' \cap \partial_+(R_{a_4})$ and $c' \cap \partial_-(R_{a_4})$ lie in $T_e$ and $B_e$, respectively, for some $e \in K = \{1, \dots, k\}$. By lemma \ref{lem:self_gluing} we have that $c'$ and $a_0$ fills. 

We now consider the case that $\delta^p \cap \partial_+(R_{a_4})$ and $\delta^q \cap \partial_+(R_{a_4})$ lie in distinct top buckets. Without loss of generality in the arguments below, we can assume that $\delta^p \cap \partial_+(R_{a_4})$ and $\delta^q \cap \partial_+(R_{a_4})$ lie in $T_p$ and $T_q$, respectively. Let the end points of $c'$ be in the top bucket $D = T_d$ and the bottom bucket $T=B_{d-l}$ where $r \leq d-l \leq d \leq s$. The set $\bigcup_{i=d-l}^{d}(T_i \cup B_i)$ forms the inside of $c$. For a $c$ with a single strand $c'$, if $c \cap a_0 = c' \cap a_0$, we recall from section \ref{sec:set} that such a $c$ is said to be a standard single strand curve. We first show that any $c$ with $i(c, a_4) = 1$ is a standard single strand curve.

If $\mathcal{T}$ is the disc which contains $T$, then there exists another top or, bottom bucket $T^*$ which contains an endpoint of the arc of $c \cap \mathcal{T}$ containing the subarc $c' \cap \mathcal{T}$. If $c$ is a standard single strand curve then $T^* = D$. If $T^* \neq D$, let $\mathcal{D}$ be the component of $S_g \setminus (a_0 \cup a_4)$ containing $D$. Let $D^*$ be the bucket in $\mathcal{D}$ where the other end of the arc in $(c \setminus a_0) \cap \mathcal{D}$ containing $c' \cap D$ lies.

\begin{lem}
\label{lem:T_D_inside}
Let $T^* \neq D$. If either $T^*$ or, $D^*$ are inside $c$ then there exists a representative of $c$ that is a standard single strand curve.
\end{lem}

\begin{proof}
Without loss of generality, suppose that $D^*$ is inside $c$. We show that $D^*$ can't be a top bucket. Similar arguments ensure that whenever $T^*$ is inside $c$, it can't be a bottom bucket. Assume on the contrary that $D^*$ is a top bucket. Then the arc of $c$ in $D^*$ is either as in figure \ref{fig:D_same_side_i} or, \ref{fig:D_same_side_ii}. In either case, consider $\gamma$ as shown in figure \ref{fig:D_same_side}. We observe that in figure \ref{fig:D_same_side}, the arc $(\gamma \setminus a_0) \cap \mathcal{D}$ from $D^*$ to $D$ is parallel to the arc $(c \setminus a_0) \cap \mathcal{D}$ from $D^*$ to $D$. By corollary \ref{lem:gamma_fill}, $\gamma$ is essential. Since, $\gamma$ is inside $c$ it implies that $\gamma$ is inside $\delta^{p,q}$. Thus, $\gamma \cap \delta = \phi$. Further, by the construction of $\gamma$, we have that $\delta$, $\gamma$, $a_4$ is a path. But this contradicts $d(\delta, a_4) \geq 3$.

\begin{figure}
     \centering
     \begin{subfigure}{0.45\textwidth}
         \centering
	\includegraphics[width=\textwidth]{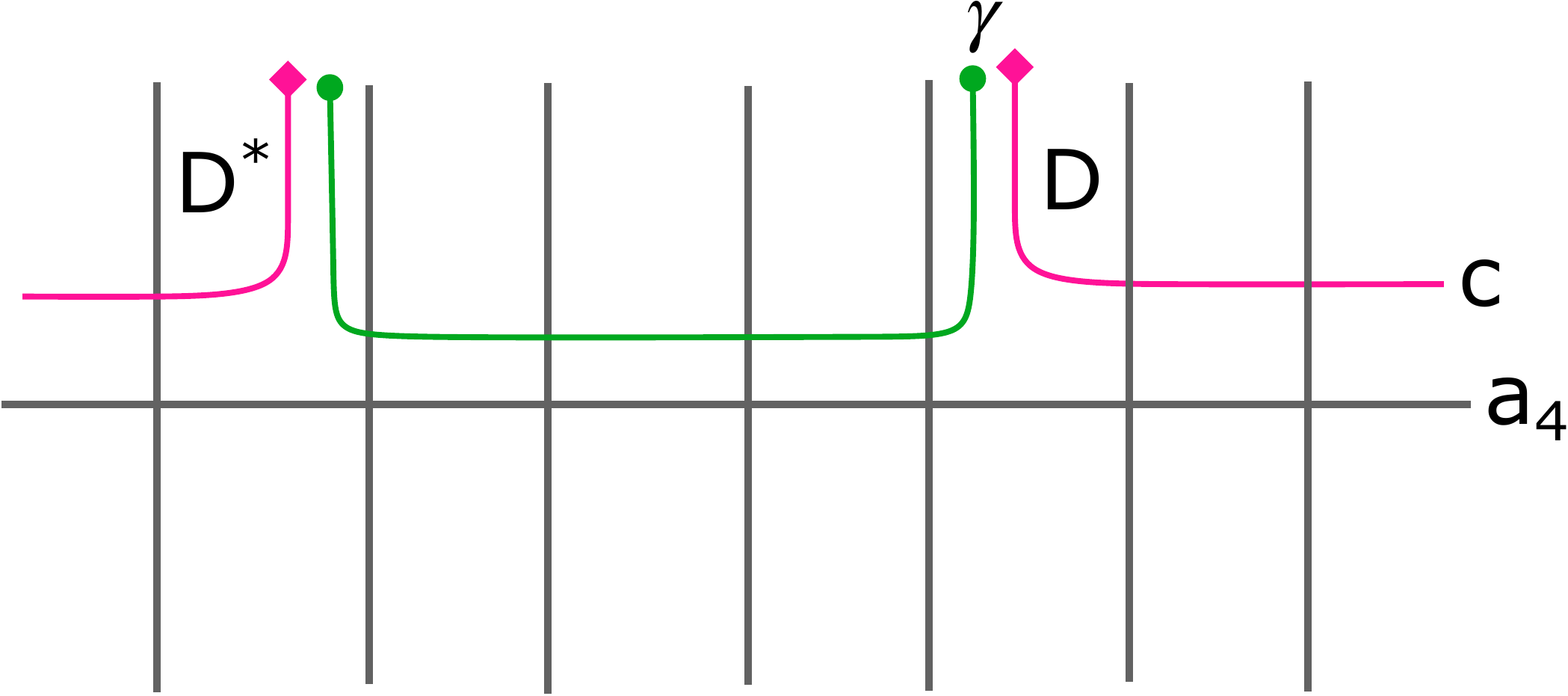}
	\caption{}
	\label{fig:D_same_side_i}
     \end{subfigure}
     \hfill
     \begin{subfigure}{0.45\textwidth}
         \centering
	\includegraphics[width=\textwidth]{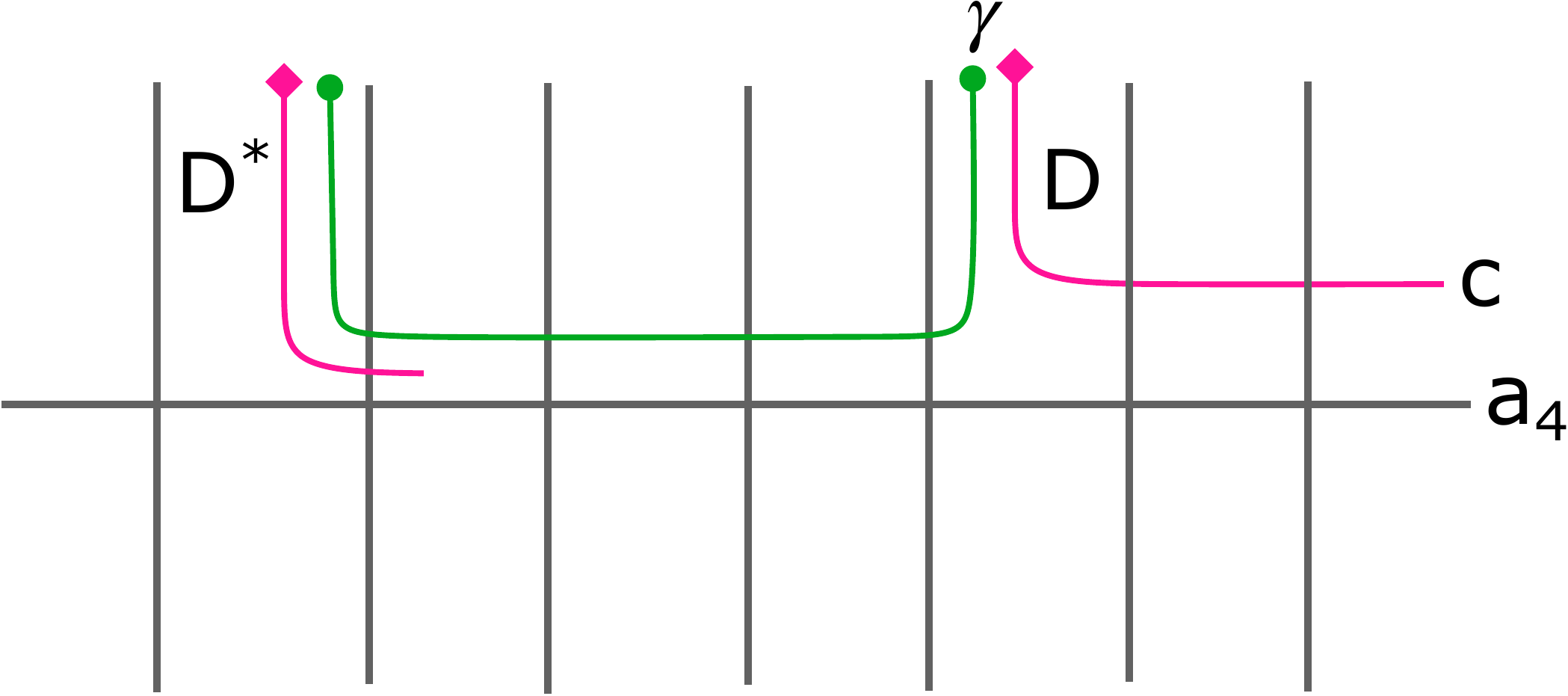}
	\caption{}
	\label{fig:D_same_side_ii}
     \end{subfigure}
     \caption{}
     \label{fig:D_same_side}
\end{figure}

We give the isotopy for the case when $T^*$ is inside $c$.  Let $c_D$ and $c_{T^*}$ be any two points on the interior of the arcs $c \cap D$ and $c \cap T^*$, respectively. Let $\tilde c$ be the component of $c \setminus \{c_D \cup c_{T^*}\}$ that contains the point $c \cap a_4$. Let $\zeta_1$ be the closed arc $c \setminus \tilde c$. Recall that $D = T_{i_0}$. If $T^* = T_{i_0 - s}$, let $\zeta_2$ be the arc passing through $T_{i_0 - s}$, $T_{i_0 - s+1}$, \dots, $T_{i_0}$ parallel to $a_4$ and having end points $c_D$ and $c_{T^*}$. Let $\zeta$ be the curve on $S_g$ formed from the concatenation of $\zeta_1$ and $\zeta_2$. 

We now show that $\zeta$ has to be a trivial curve on $S_g$. By construction, we have $\zeta \cap a_4 = \phi$. Since $\zeta_1$ is an arc of $c$, $\zeta_1 \cap \delta = \phi$. Since $T^*$ is inside $c$ and hence, inside $\delta^{p,q}$, it follows that $\zeta_2 \cap \delta = \phi$. Thus, $\zeta \cap \delta = \phi$. If $\zeta$ is non-trivial it contradicts the fact that $d(\delta, a_4) \geq 3$.
 
Since $\zeta$ is trivial, we have that $\zeta_1$ is isotopic to $\zeta_2$ by an isotopy, say $L'$. We perform an isotopy, $L$, of $c$ such that $L(\tilde c) = \tilde c$ and $L(c \setminus \tilde c) = L'(\zeta_1) = \zeta_2$. Thus $L(c)$ is a standard single strand curve on $S_g$ whose strand has its end points in $T$ and $T^*$. 

A similar proof follows for the case if $D^*$ is inside $c$ and $T^*$ is outside $c$ by reversing the roles of $T^*$ with $D^*$ and $D$ with $T$.

\begin{figure}
     \centering
     \begin{subfigure}{\textwidth}
         \centering
	\includegraphics[scale=0.32]{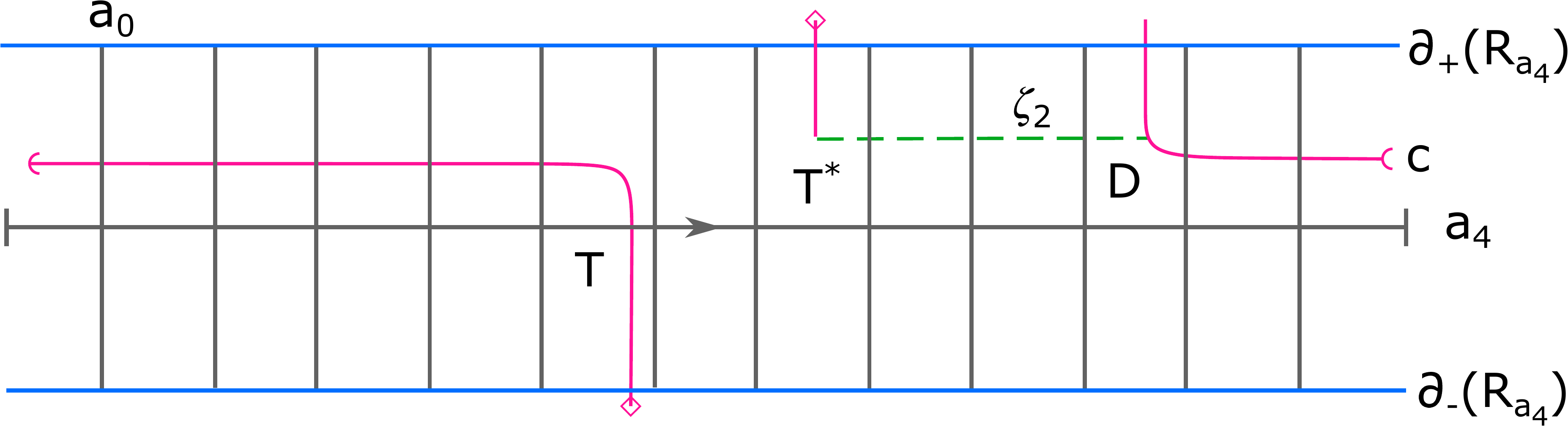}
	\caption{}
	\label{fig:zeta_i}
     \end{subfigure}
     \hfill
     \begin{subfigure}{\textwidth}
         \centering
	\includegraphics[scale=0.32]{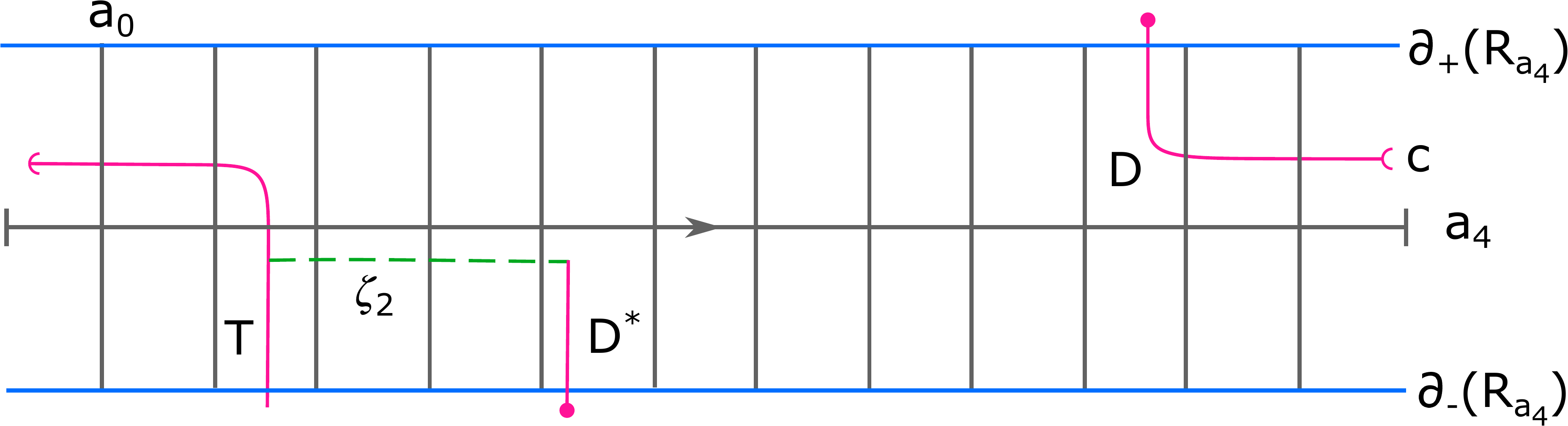}
	\caption{}
	\label{fig:zeta_ii}
     \end{subfigure}
     \caption{}
     \label{fig:zeta}
\end{figure}

\end{proof}

\begin{lem}
\label{lem:nonrect}
If $c$ is a standard single strand curve then, there exists a representative of $c$ such that the end points of its strand lies in a top and a bottom bucket of a non-rectangular component of $S_g \setminus (a_0 \cup a_4)$.
\end{lem}

\begin{proof}
Let us suppose $\mathcal{T}$ is a rectangle. From the notations defined in the previous section, it follows that the four vertices of $\mathcal{T}$ are $w_{d}$, $w_{d+1}$, $w_{d-l}$ and $w_{d-l+1}$. Since $c$ has a single strand, we have that $i(c, a_4) =1$. As $c$ and $a_0$ are in minimal position, the two parallel edges corresponding to the $a_0$-arcs in $\mathcal{T}$ are as follows : one edge is between $w_{d-l}$ and $w_d$ and the other edge is between $w_{d-l+1}$ and $w_{d+1}$. A schematic of the bigon formed between $c$ snd $a_0$ if the $a_0$-edges in $\mathcal{T}$ are otherwise is shown in figure \ref{fig:lemma_6}. We note that, since $\delta \cap c = \phi$, we have that $c \cap \partial_+(R_{a_4})_{[d, d+1]} = \phi$. This means that $D=T_d$ and $T_q$ don't coincide. 

\begin{figure}
\centering
\includegraphics[scale=0.5]{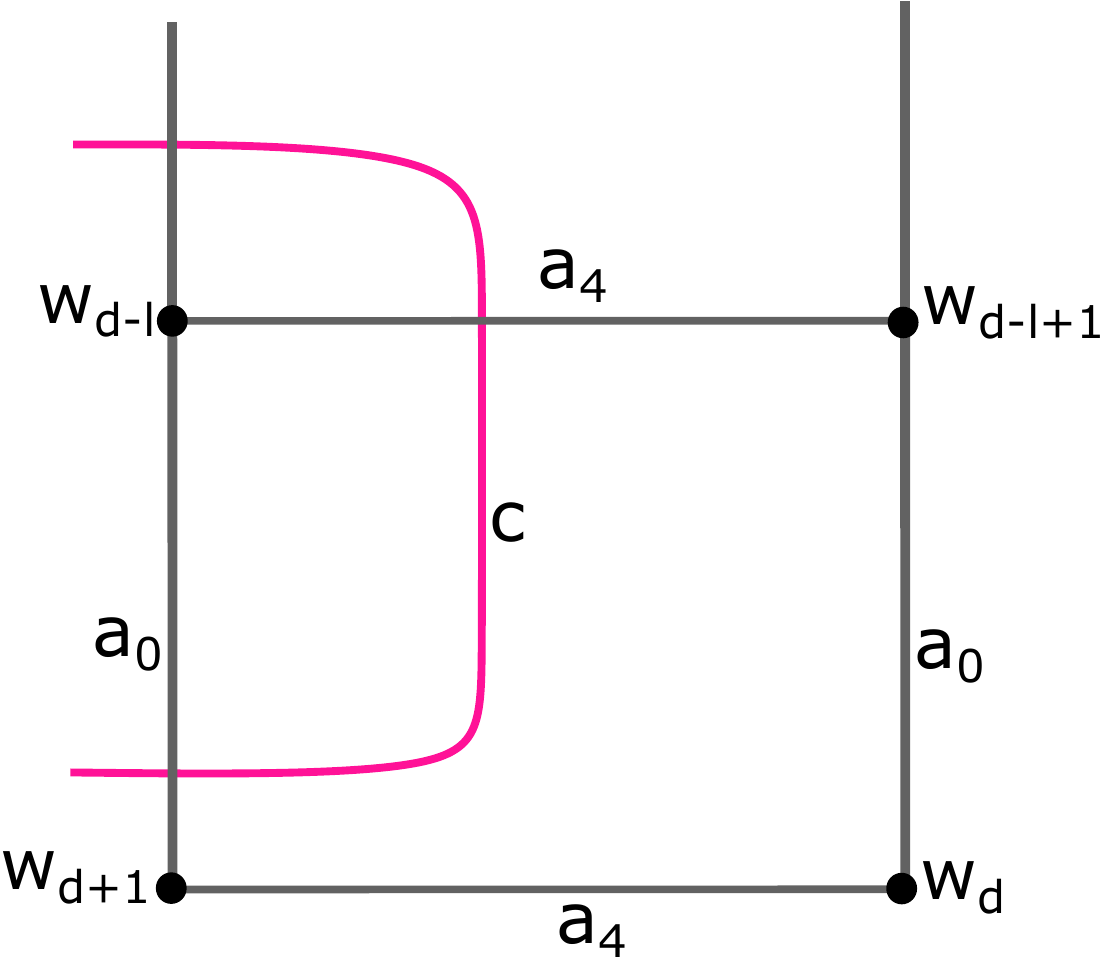}
\caption{The possible bigon formed if the $a_4$ edges in $\mathcal{T}$ aren't parallel.}
\label{fig:lemma_6}
\end{figure}

We first show that there exists $j$ such that $d \leq j \leq q$ and $T_j$ is a top bucket of some non-rectangular component of $S_g \setminus(a_0 \cup a_4)$. On the contrary assume that $T_j$ for every $d \leq j \leq q$ is a top bucket of some rectangular disc. Let $T_q$ be contained in the rectangular disc $\mathcal{R}$ and let the component of $\delta^q \cap \mathcal{R}$ with end points on $a_0$ and $\partial_+(R_{a_4})$ be $\delta^q_*$. Then the $a_4$-edges in $\mathcal{R}$ are $a_{4_{[w_q, w_{q+1}]}}$ and $a_{4_{[w_{q-l}, w_{q-l+1}]}}$. The union of the rectangular components of $S_g \setminus(a_0 \cup a_4)$ containing $T_j$ for every $d \leq j \leq q$ is again a rectangle, say $R$. In particular, $\mathcal{T}$ and $\mathcal{R}$ are contained in $R$. Since the $a_4$ edges of $\mathcal{T}$ are in oriented parallely, it gives that the $a_4$ edges of $R$ are also oriented parallely. In particular, the $a_4$ edges of $\mathcal{R}$ are oriented parallely. As a result $B_{q-l} \subset \mathcal{R}$. Thus the component of $\delta \cap \mathcal{R}$ that contains $\delta^q_*$ has an end point on $a_{4_{[w_{q-l}, w_{q-l+1}]}}$, say $x$. A schematic of the above description is as in figure \ref{fig:rect_iso}. Since $\delta^q \setminus a_0$ covers every arc of $a_4 \setminus a_0$ except $a_{4_{[w_q, w_{q+1}]}}$, there is an arc of $\delta^q$ parallel to $a_{4_{[w_{d-l}, w_{q-l+1}]}}$ in $R_{a_4}$. Since there are no points of $\delta \cap a_0$ on $a_0^{q-l}$ between $\delta^q \cap a_0^{q-l}$ and $w_{q-l}$, there is no possibility of $x$ joining to any arc of $\delta \setminus a_0$. Thus, we have that there is a $T_j$ for some $d \leq j \leq q$ such that $T_j$ is contained in some non-rectangular component of $S_g \setminus (a_0 \cup a_4)$.

\begin{figure}
\centering
\includegraphics[scale=0.17]{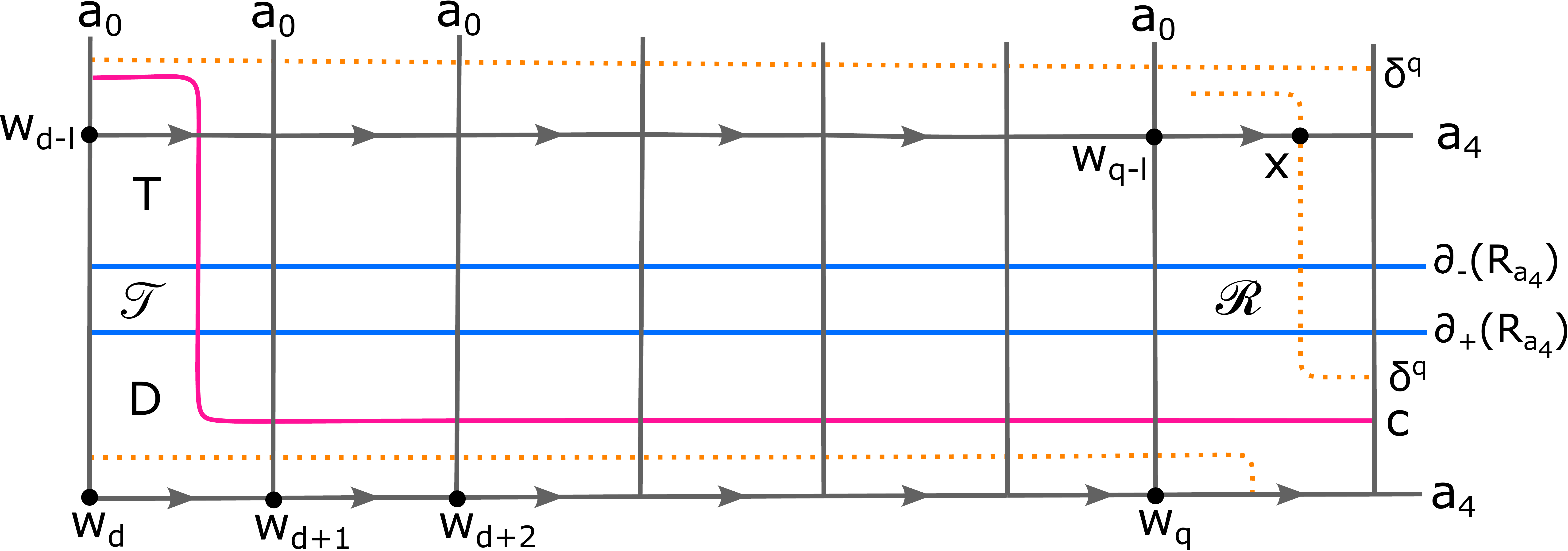}
\caption{A schematic of $c$}
\label{fig:rect_iso}
\end{figure}

Consider the non-rectangular disc $H$ such that there is a top bucket $T_{d_0}$ of $H$ where $d < d_0 \leq q$ and $T_j$ for $d \leq j < d_0$ are top buckets of rectangular discs. Since $T_j$ are rectangles for $d \leq j < d_0$, we have that $B_{d-l+i}$ for $0 \leq i < d_0 -d$ are bottom buckets of rectangles. Further, for $0 \leq i < d_0-d$, $T_{d+i}$ and $B_{d-l+i}$ are buckets of the same rectangular disc. Consider an isotopy, $L$, of $c$ that moves the point $c \cap a_4$ along the increasing direction of $a_4$ from $a_{4_{[w_{d-l}, w_{d-l+1}]}}$ to $a_{4_{[w_{d_0-l}, w_{d_0-l+1}]}}$ such that $L(c)$ is in minimal position with $a_0$ and $a_4$. A schematic of $L$ is shown in figure \ref{fig:rect_iso_L}. Since $T_{d_0}$ is in the same $\delta$-track as $D$ and $T$, $L(c)$ remains to be in minimal position with $\delta$. 
\end{proof}

\begin{figure}
\centering
\includegraphics[scale=0.32]{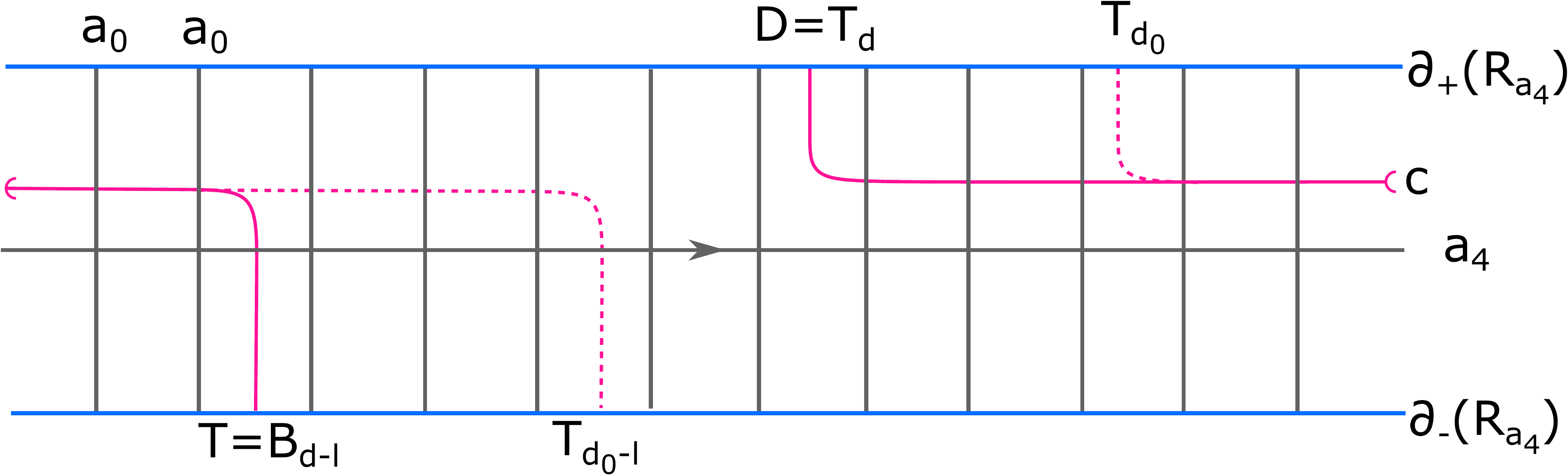}
\caption{The dotted line is a schematic of $L(c)$}
\label{fig:rect_iso_L}
\end{figure}

As a consequence of lemma \ref{lem:nonrect}, we can now assume that the strand of $c$ has its end points in a top and bottom buckets, $A_t$, $A_b$ of a non-rectangular disc, say $H$. If $H$ is a $2n$-gon with $n \geq 4$, then by lemma \ref{lem:bucket_lemma} there exists either a top or, bottom bucket $A$ of $H$ outside the delta track $\delta^{p,q}$. We will assume that $A$ is a top bucket. When $A$ is a bottom bucket, similar conclusions can be made about $c$ and $a_0$ by interchanging the roles of $A_t$ and $A_b$ in the following arguments. 

Let $\gamma$ be a scaling curve from $A_t$ to $A$. Let $\gamma'$ be the arc in $\gamma \cap H$ which contains $(\gamma \cap A) \cup (\gamma \cap A_t)$. By corollary \ref{lem:gamma_fill}, we have that $\gamma$ and $a_0$ fill. The construction of $\gamma$ gives that every arc in $(\gamma \setminus a_0) \setminus \gamma'$ is covered by some arc in $c \setminus a_0$. Let $\Lambda_1$ and $\Lambda_2$ be the two polygonal components of $S_g \setminus (a_0 \cup \gamma)$ containing the two edges corresponding to $\gamma'$. Let $\Lambda$ be the component formed by gluing $\Lambda_1$ and $\Lambda_2$ along $\gamma'$. If $\Lambda_1$ and $\Lambda_2$ are distinct, then $\Lambda$ is a disc. It then follows that the components of $(S_g \setminus a_0) \setminus c$ correspond to the components of $\{(S_g \setminus a_0) \setminus (\gamma \setminus \gamma')\} \cup \Lambda$. Thus, $c$ and $a_0$ fill whenever $\Lambda_1 \neq \Lambda_2$. If $\Lambda_1 = \Lambda_2$, then $\Lambda$ is an annulus. Note that by construction of $\gamma$ there exists arcs of $a_4 \setminus a_0$ that covers every arc of $(\gamma \setminus a_0) \setminus \gamma'$. Consider a representative of $a_4$ such that for every arc of $(\gamma \setminus a_0) \setminus \gamma'$, the respective arc of $a_4 \setminus a_0$ which covers it also overlaps it. Let $\mathcal{P}$ be the central curve of the annulus $\Lambda$. We have that $\mathcal{P}$ will be an essential curve on $S_g$. If not, $a_0$ ceases to be connected. Let the two boundary components of $\Lambda$ be $\partial_+(\Lambda)$ and $\partial_-(\Lambda)$.

Let $Y$ be a component of $S_g \setminus (a_0 \cup a_4)$ such that one of the arcs in $(a_4 \setminus a_0) \cap Y$, say $y$, has one of its end points on $\partial_+(\Lambda)$ and another on $\partial_-(\Lambda)$. We have that $(a_0 \cap a_4) \cap Y \subset \partial(\Lambda)$. Since $\gamma \cap a_4 = \phi$, the arcs in $(a_4 \setminus a_0) \cap Y$ are either in the interior of $\Lambda$ or, in $\partial(\Lambda)$. Thus, $Y$ is a polygon in the annulus $\Lambda$ such that its edge $y$ is in the interior of $\Lambda$. It follows that there is another arc in $(a_4 \setminus a_0) \cap Y$ with its end points on $\partial_+(\Lambda)$ and $\partial_-(\Lambda)$.

Since $\mathcal{P}$ is a curve on $S_g$ with $\mathcal{P} \cap a_0 = \phi$ and $a_0$ and $a_4$ fill $S_g$, there must exist an arc $y_1$ in $a_4 \setminus a_0$ such that $y_1 \cap \mathcal{P} \neq \phi$. As $\mathcal{P}$ is the core curve of $\Lambda$, the end points of $y_1$ must lie on $\partial_+(\Lambda)$ and $\partial_-(\Lambda)$. Let $Y_1$ be one of the discs of $S_g \setminus (a_0 \cup a_4)$ containing an edge corresponding to $y_1$. By the argument in the previous paragraph, there exists another arc, $y_2$ in $a_4 \setminus a_0$ with end points on $\partial_+(\Lambda)$ and $\partial_-(\Lambda)$. Let $Y_2$, if exists, be the other disc which contains the other edge corresponding to $y_2$. We apply this process inductively to obtain all the discs $Y_1$, $Y_2$, \dots, $Y_l$ with edges $y_1$, $y_2$, \dots, $y_l$ having end points on distinct components of $\partial(\Lambda)$. If any $y_i$ is not inside $c$ we have that an arc of $c \setminus a_0$ covers this particular $y_i$. Thus, $c$ and $a_0$ fill $S_g$. If all $y_i's$ lie inside $c$, then $\taf$, $\delta$, $c$, $\mathcal{P}$, $a_0$ is a geodesic of distance $4$. If such a geodesic of length $4$ exists with $l=4$, a schematic of $R_{a_4}$ and $\mathcal{P} \cap R_{a_4}$ is as in figure \ref{fig:possible_p} upto renaming of the components $Y_i$ for $1 \leq i \leq 4$. 

\begin{figure}
\centering
\includegraphics[scale=0.32]{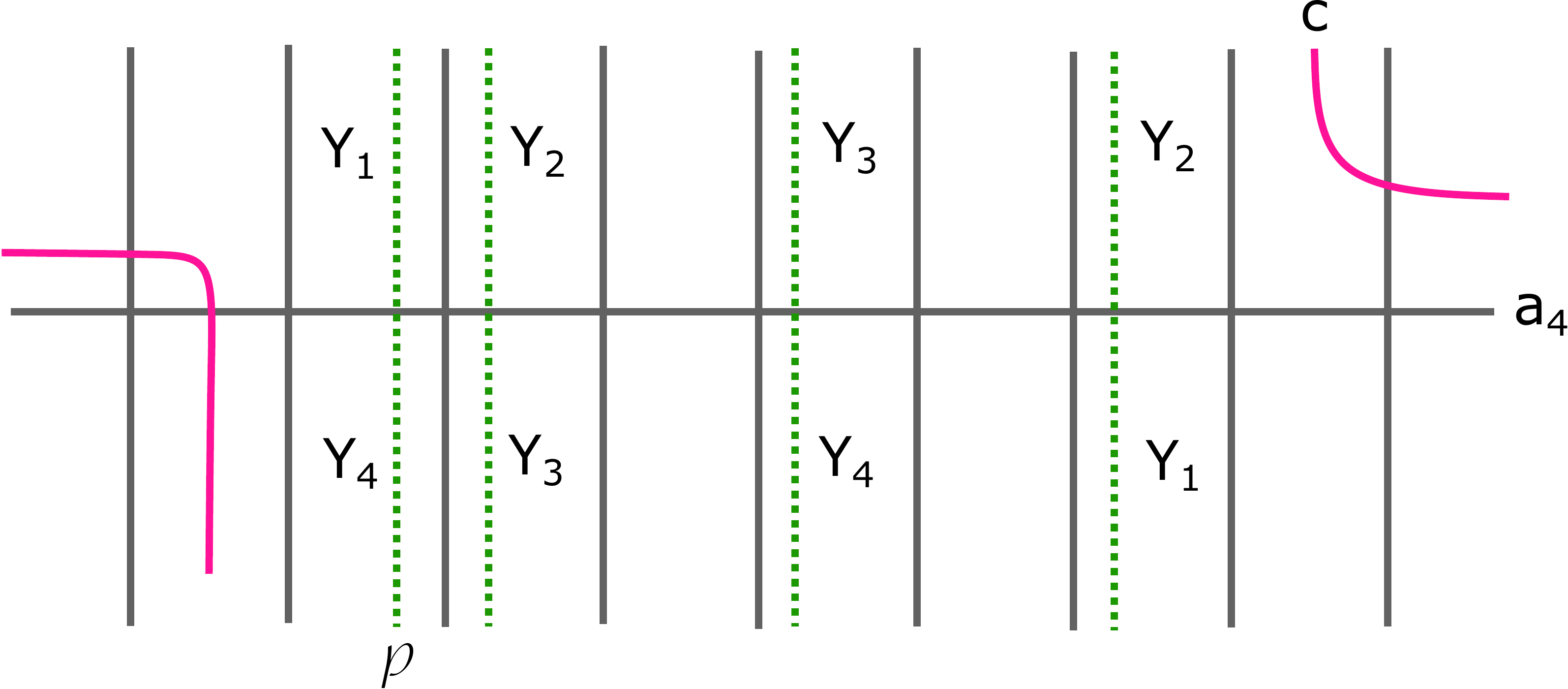}
\caption{A schematic of $\mathcal{P}$ (dotted lines) if $Y_1$, \dots, $Y_4$ occur as above in $R_{a_4}$.}
\label{fig:possible_p}
\end{figure}

If $i(c, a_4) = 1$ but $c$ is not a standard single strand curve, by lemma \ref{lem:T_D_inside} neither $T^*$ nor $D^*$ are inside $c$. Then $T^*$ can be either a top or, bottom bucket. If $T^*$ is a bottom bucket, consider the scaling curve $\gamma$ as in \ref{fig:gamma_ai} and let the arc in $\gamma \setminus a_0$ between $T$ and $T^*$ be $\gamma'$. It can be seen from \ref{fig:gamma_ai} that by virtue of the choice of $\gamma$, $c \setminus a_0$ contains a subset of arcs that cover all the arcs in $(\gamma \setminus a_0) \setminus \gamma'$. Since $T^*$ doesn't lie inside $c$, there exists a pair of arcs in $c \setminus a_0$ that almost covers $\gamma'$. This pair of arcs comprises of the following two arcs of $c \setminus a_0$ : firstly the arc that contains the subarc $T^* \cap c$ and secondly, the arc that lies in the top bucket which has a common component of $a_4 \setminus a_0$ with $T^*$. This second arc of $c \setminus a_0$ covers the edge corresponding to $a_4 \cap T^*$. Thus by lemma \ref{lem:almost_fill}, $c \setminus a_0$ forms a filling system of arcs on $S_g \setminus a_0$. If $D^*$ is a top bucket, we construct $\gamma$ as in figure \ref{fig:gamma_aii} and a similar argument as above gives that $c$ and $a_0$ fill $S_g$. If $T^*$ is a top bucket and $D^*$ is a bottom bucket then construct $\gamma$ as in figure \ref{fig:gamma_aiii}. Similar arguments as above along with lemma \ref{lem:almost_fill} concludes that $c$ and $a_0$ fill $S_g$.

\begin{figure}
     \centering
     \begin{subfigure}{\textwidth}
         \centering
	\includegraphics[scale=0.32]{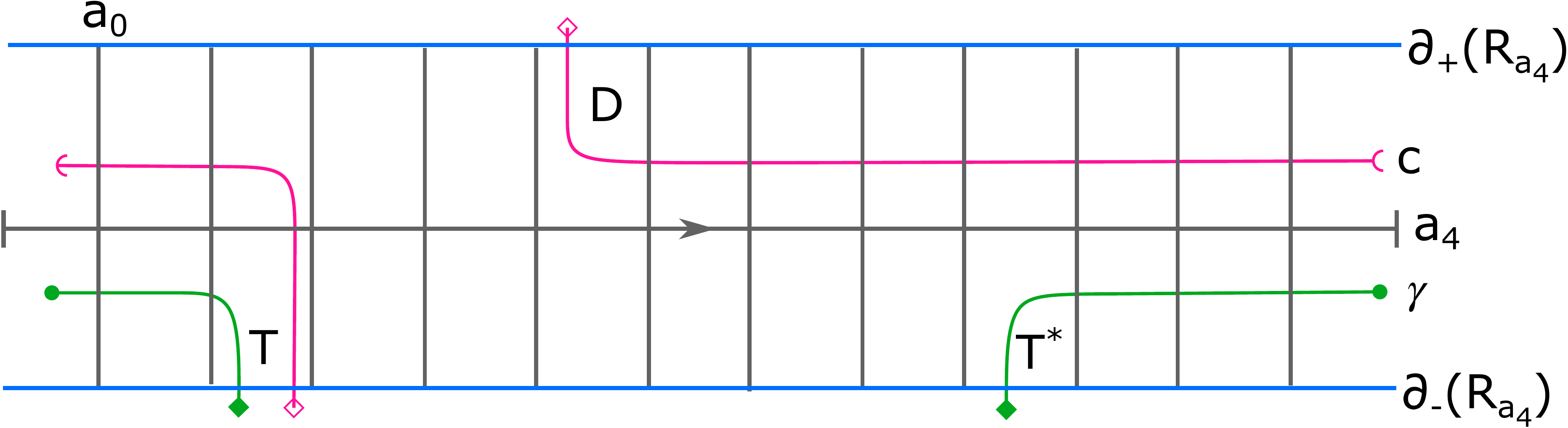}
	\caption{}
	\label{fig:gamma_ai}
     \end{subfigure}
     \hfill
     \begin{subfigure}{\textwidth}
         \centering
	\includegraphics[scale=0.32]{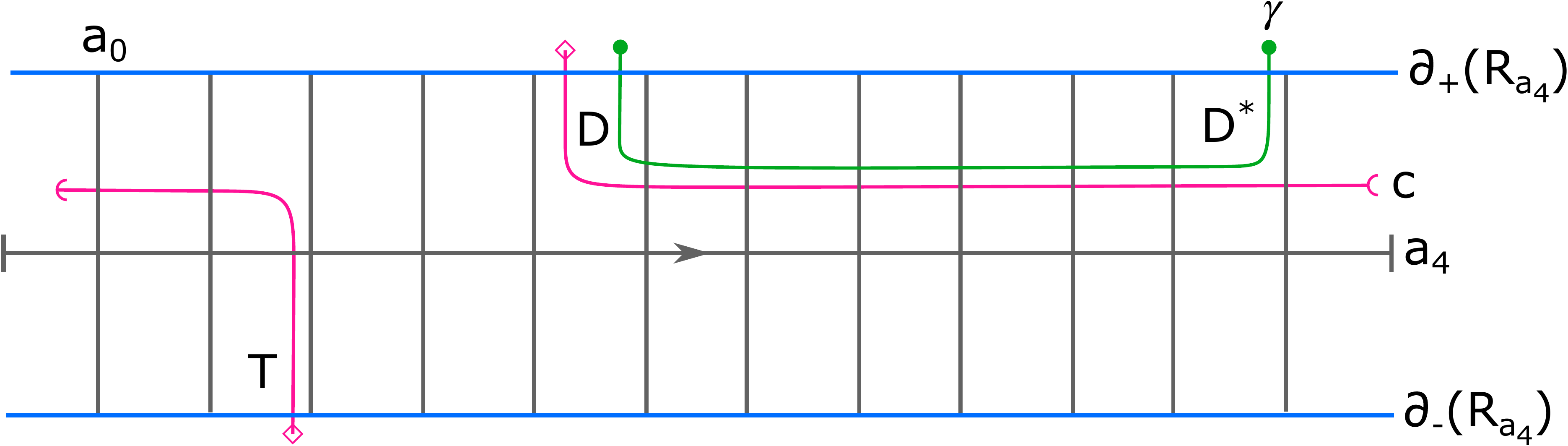}
	\caption{}
	\label{fig:gamma_aii}
     \end{subfigure}
     \caption{}
     \label{fig:gamma_a}
\end{figure}

\begin{figure}
\centering
\includegraphics[scale=0.32]{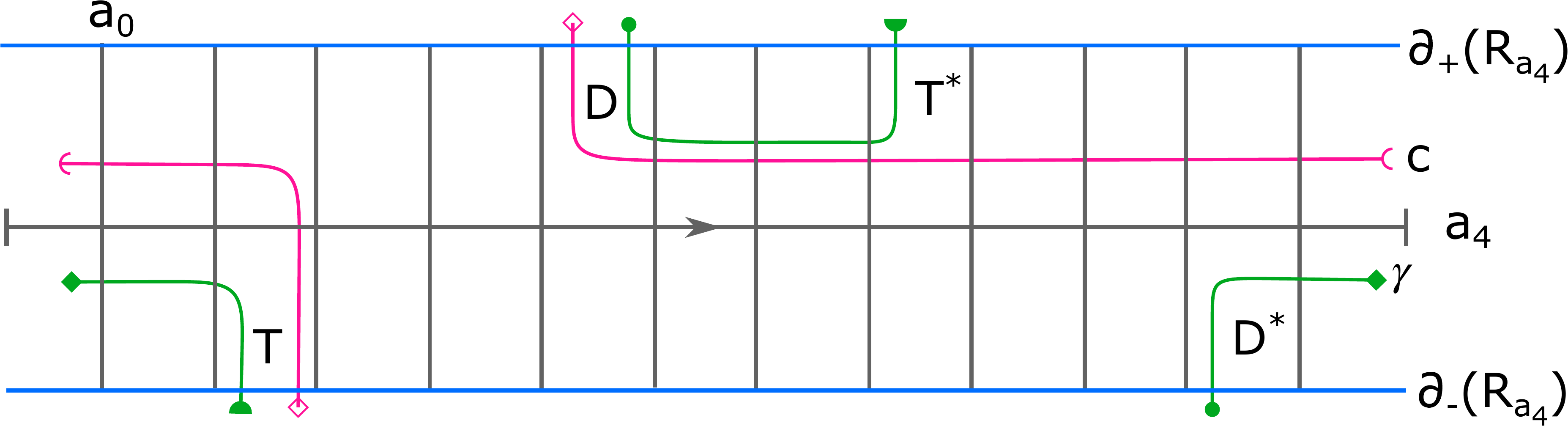}
\caption{}
\label{fig:gamma_aiii}
\end{figure}

\noindent \textbf{Case ii :} Suppose there exists at least two distinct strands, $c'$ and $c''$, of $c$ in $R_{a_4}$. Assume that $c' \cap \partial_+(R_{a_4})$ lie on $\partial_+(R_{a_4})_{[\delta^p, \delta^q]}$ and $c'' \cap \partial_+(R_{a_4})$ lie on $\partial_+(R_{a_4})_{[\delta^r, \delta^s]}$ where $1\leq r< s \leq m$. If either $\delta^p \cap \partial_+(R_{a_4})$ and $\delta^q \cap \partial_+(R_{a_4})$ lie in the same top bucket or, $\delta^r \cap \partial_+(R_{a_4})$ and $\delta^s \cap \partial_+(R_{a_4})$ lie in the same top bucket then by lemma \ref{lem:self_gluing}, $c$ and $a_0$ fills $S_g$. The argument is similar to that of case i when the end points of the strand lie in $T_e$ and $B_e$ for some $e \in K = \{1, \dots, k\}$. We can thus assume that $p \leq r$.

Suppose $p < r$. If $q < r$, then for every $j \in K$, there exists an arc of $(c' \cup c'') \setminus a_0$ that covers $a_4^j$. Thus, $c \setminus a_0$ forms a filling system of arcs in $S_g \setminus a_0$. If $q=r$, then figure \ref{fig:double_strand_c} gives the only instance when there exists a $J \in K$ such that $a_4^J$ isn't covered by an arc of $(c' \cup c'') \setminus a_0$. It follows from lemma \ref{lem:self_gluing} that $c \setminus a_0$ forms a filling system of arcs in $S_g \setminus a_0$.

\begin{figure}
\centering
\includegraphics[scale=0.6]{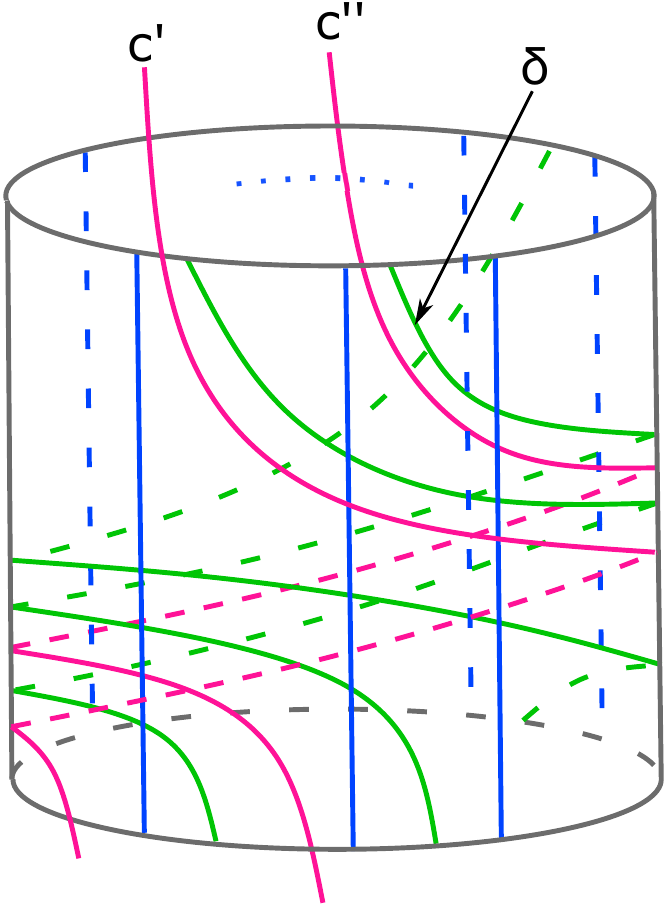}
\caption{A schematic of $c'$ and $c''$ in $R_{a_4}$ when they do not cover every arc of $a_4 \setminus a_0$.}
\label{fig:double_strand_c}
\end{figure}

Suppose $p=r$. We rename $c'$ to be the strand of $c$ such that one of the components of $\partial_+(R_{a_4}) \setminus \{\delta^p, c'\}$ doesn't contain any points of $\partial_+(R_{a_4}) \cap c$ in its interior. Let $z$ be the component of $c \setminus R_{a_4}$ containing $c' \cap \partial_-(R_{a_4})$. Rename $c''$ to be the strand of $c$ such that $z \cap c'' \neq \phi$. We claim that $z \cap c''$ lies on $\partial_+(R_{a_4})$. Let $z'$ be the component of $\partial_-(R_{a_4}) \setminus z$ such that no points of $\partial_-(R_{a_4}) \cap c$ lies in its interior. On the contrary, if $z \cap c''$ lies on $\partial_-(R_{a_4})$ then we get that $z$ is isotopic to $z'$. This follows because the curve obtained by concatenating $z$ and $z'$ is a curve disjoint from $\delta$ and $a_4$. Since $d(\delta, a_4) \geq 3$, this curve has to be non-essential. Therefore, we must have that $z \cap c''$ lies on $\partial_+(R_{a_4})$. Let $\tilde z$ be the component of $c \setminus R_{a_4}$ containing $c' \cap \partial_+(R_{a_4})$. A similar argument ensures that the end point of $\tilde z$ lies in $\partial_-(R_{a_4})$. Thus, following the naming convention of $D$, $D^*$, $T$ and $T^*$ for the strand $c'$ as in case i, we can consider a scaling curve as in figure \ref{fig:gamma_aiii}. Similar arguments as in case i gives that $c$ and $a_0$ fills $S_g$.

\section{Conclusion}
\label{sec:conc}
If $d(a_0, \taf)=4$, then there exists $\delta \in B_1(\taf)$ and corresponding to $\delta$ there exists $c \in B_1(\delta) \cap B_2(\taf)$ such that $a_0$, $p$, $c$, $\delta$, $\taf$ is a geodesic. We consider the representatives of $\delta$ and $c$ to be the ones as described in section \ref{sec:main}. Then a schematic of a possible $p$ in $R_{a_4}$ is as in figure \ref{fig:possible_p}. We now describe an equivalent condition for the existence of $p$ in the form of buckets. Given such a curve $p$, consider the collection $\mathcal{Y}_p$ of all top and bottom buckets in $R_{a_4}$ containing $p$. Since $p \cap a_0 = \phi$, if $T_i \in \mathcal{Y}_p$ for some $i \in \{1, \dots, k\}$ then $B_i \in \mathcal{Y}_p$. Conversely, we define a collection of pairs of top and bottom bucket $\{(T_i, B_i)\}_{i \in I}$ for some $I \subset K = \{1, \dots, k\}$ where for every $T_i$ there exists unique $j \in I$ and $j \neq i$ such that $T_i \cup T_j \subset Y$ (or, $T_i \cup B_j \subset Y$) for some component $Y$ of $S_g \setminus (a_0 \cup a_4)$ as a \textit{stack of buckets}. We note that given a stack of buckets we can always construct a curve disjoint from $a_0$. The pattern in figure \ref{fig:possible_p} can be described as the inside of $c$ containing a stack of buckets. For any given $c \in B_2(\taf)$ if the inside of $c$ contains a stack of buckets, we say \textit{$c$ has the stacking property}. Thus, we conclude that :
\begin{lem}
\label{lem:geq4}
$d(a_0, \taf) = 4$ if and only if there exists $c \in B_2(\taf)$ such that $c$ has the stacking property.
\end{lem}

From our analysis of curves $c \in B_2(\taf) \cap B_1(\delta)$ in section \ref{sec:main}, we have that if $c$ is not a standard single strand curve then $c$ and $a_0$ always fill. Thus we have the following theorem :

\begin{thm}
\label{thm:geq5}
Let $a_0$ and $a_4$ be curves on $S_g$ such that $d(a_0, a_4) = 4$ and the components of $S_g \setminus (a_0 \cup a_4)$ doesn't contain any hexagons. Then, $d(a_0, \taf) \geq 5$ if and only if there doesn't exist any standard single strand curve $c \in B_2(\taf)$ having the stacking property.
\end{thm}

An advantage of theorem \ref{thm:geq5} is that it reduces the number of possible vertices through which a path of length $4$ between $a_0$ and $\taf$ if it exists can pass.

\section{A pair of curves at a distance $5$ in $\mathcal{C}(S_2)$}
\label{sec:eg5}

Let $a_0$ and $a_4$ be curves on $S_2$ as in figure \ref{fig:a0a4}. These curves are at a distance $4$ in $\mathcal{C}(S_2)$ and are taken from \cite{BMM}. In this section we show that $d(a_0, \taf)=5$ by giving a geodesic between them. Let $b_0 = a_0$, $b_1$, $b_2$, $b_3$ be curves on $S_2$ as in figure \ref{fig:b0b3} and $b_4$ be as in figure \ref{fig:b3b4}. The juxtaposition of the curves in figure \ref{fig:b0b3} and \ref{fig:b3b4} shows that $b_0$, $b_1$, $b_2$, $b_3$, $b_4$ form a path of length $4$ in $\mathcal{C}(S_2)$. 

Since $a_0$ and $a_4$ fill $S_2$, we can give a schematic of $S_2$ by giving the components of $S_2 \setminus (a_0 \cup a_4)$ as polygons whose vertices are the points of $a_0 \cap a_4$ marked as in figure \ref{fig:a0a4} and edges correspond to arcs of $a_0 \setminus a_4$ or, $a_4 \setminus a_0$. Figure \ref{fig:R1} - \ref{fig:H3H4} represent all polygons but the rectangle with vertices $10$, $9$, $4$, $5$ of $S \setminus (a_0 \cup a_4)$. We give a juxtaposition of the curves $b_4$ and $\taf$ in minimal position on $S_2$ by giving their arcs on the polygonal discs of $S \setminus (a_0 \cup a_4)$. Since the representatives of $b_4$ and $\taf$ that we pick don't have any arcs in the rectangle of $S_2 \setminus(a_0 \cup a_4)$ with vertices $10$, $9$, $4$, $5$, we exclude this rectangles from the figures.  In figure \ref{fig:R1} - \ref{fig:H3H4}, the straight lines correspond to the arcs of $\taf$ and the dotted ones correspond to $b_4$. Since there is no intersection between these arcs, we conclude that $b_4$ and $\taf$ are at a distance $1$ in $S_2$.

We now show that $d(a_0, \taf) > 4$ by using lemma \ref{lem:geq4}. Consider the curves $\gamma_1$, $\gamma_2$, $\gamma_3$ and $\gamma_4$ as in figure \ref{fig:BMM_a4} which are at a distance $1$ from $a_4$. If for any $i_0 \in \{1,2,3,4\}$, $T_{a_4}^{-1}(\delta) \cap \gamma_{i_0} = \phi$ then $a_0$, $T_{a_4}^{-1}(\delta)$, $\gamma_{i_0}$, $a_4$ will form a path of length $3$, which is absurd. Thus, $d(T_{a_4}^{-1}(\delta), \gamma_i) \geq 2$ for $i = 1, 2, 3, 4$. Now, since $d(T_{a_4}^{-1}(\delta), a_0) = 1$ and $T_{a_4}^{-1}(\delta) \cap a_4 \neq \phi$, the arcs in the non-empty set $T_{a_4}^{-1}(\delta) \cap R_{a_4}$ are parallel to the arcs in $a_0 \cap R_{a_4}$. Since, $T_{a_4}^{-1}(\delta) \cap \gamma_i \neq \phi$ for every $i = 1, 2, 3, 4$, we refer to figure \ref{fig:BMM_a4} and observe that for any two possible consecutive arcs of $T_{a_4}^{-1}(\delta) \cap R_{a_4}$  there are no stack of buckets between them. We note that we can circumvent verifying the above for the set of all possible consecutive arcs of $T_{a_4}^{-1}(\delta) \cap R_{a_4}$ by looking at only the consecutive arcs of $T_{a_4}^{-1}(\delta) \cap R_{a_4}$ that has the maximum number of top buckets between them. Since the inside of a $c$ is contained in some $\delta$-track and the strands of $\delta$ that constitute the boundary of a $\delta$-track are $T_{a_4}$-image of some arc of $T_{a_4}^{-1}(\delta)$ therefore, no $c$ has the stacking property. Thus, $d(a_0, \taf) > 4$.
 
From the above discussion, we conclude that the path in $\mathcal{C}(S_2)$ comprising of vertices $b_0=a_0$, $b_1$, $b_2$, $b_3$, $b_4$, $b_5 = \taf$ is a geodesic of length $5$ in $\mathcal{C}(S_2)$. As an application of this example we give an upper bound on $i_{min}(2,5)$ as follows : 

\begin{cor}
$i_{min}(2,5) \leq 144$.
\end{cor}

\begin{figure}
\centering
\includegraphics[width=\textwidth]{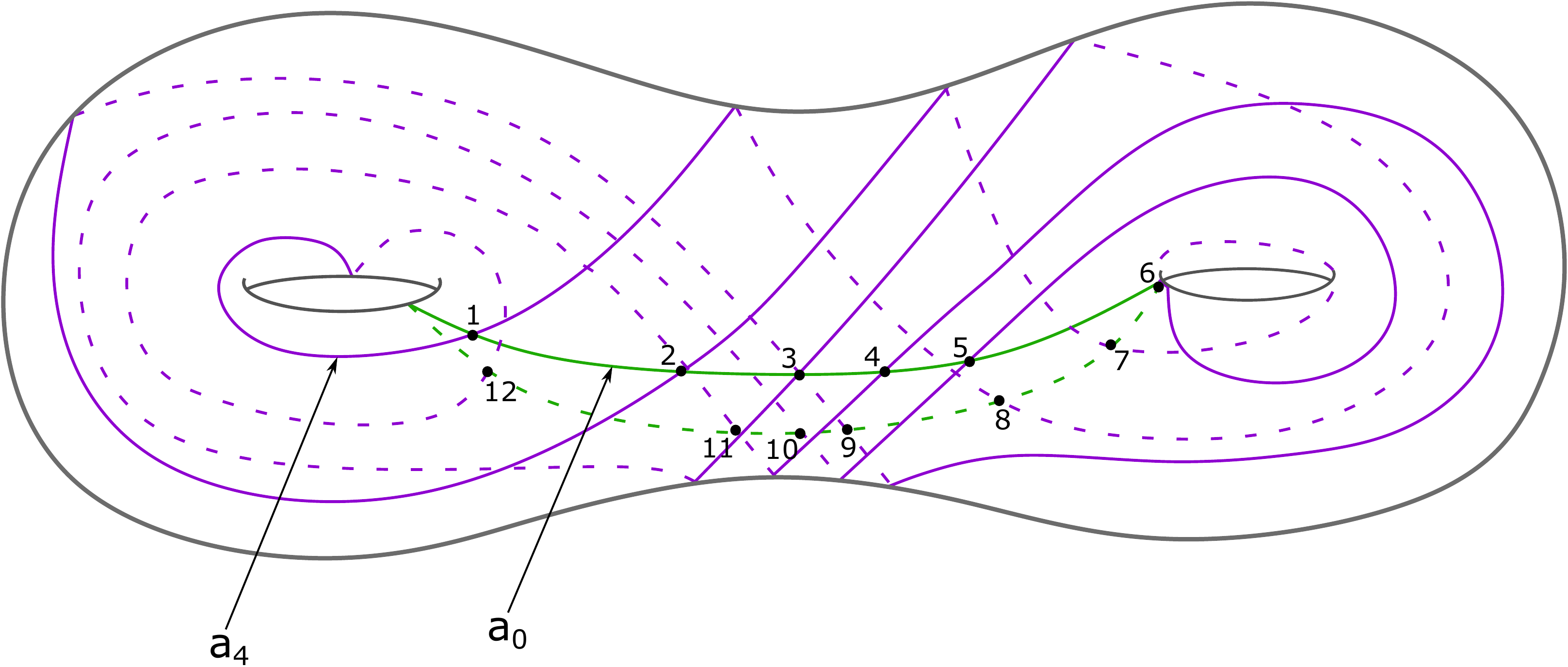}
\caption{}
\label{fig:a0a4}
\end{figure}

\begin{figure}
\centering
\includegraphics[width=\textwidth]{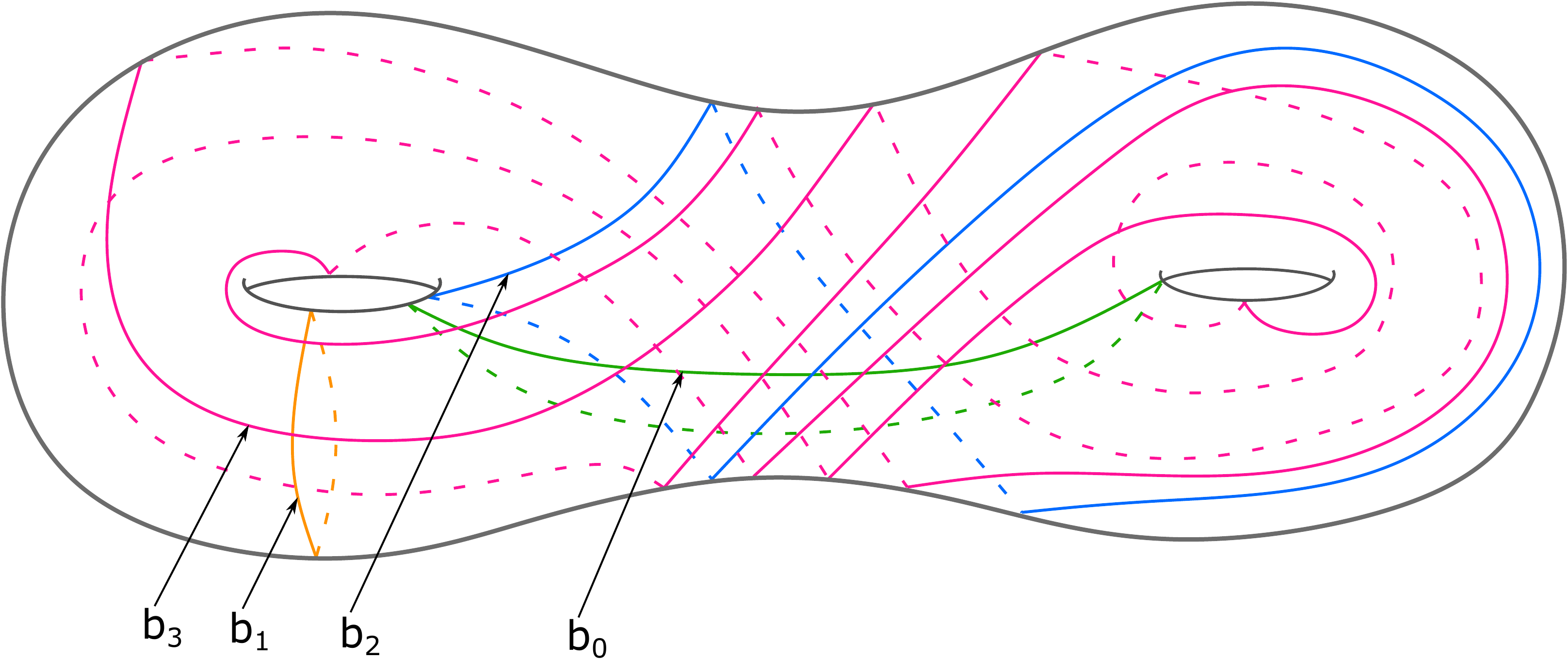}
\caption{}
\label{fig:b0b3}
\end{figure}

\begin{figure}
\centering
\includegraphics[scale=0.32]{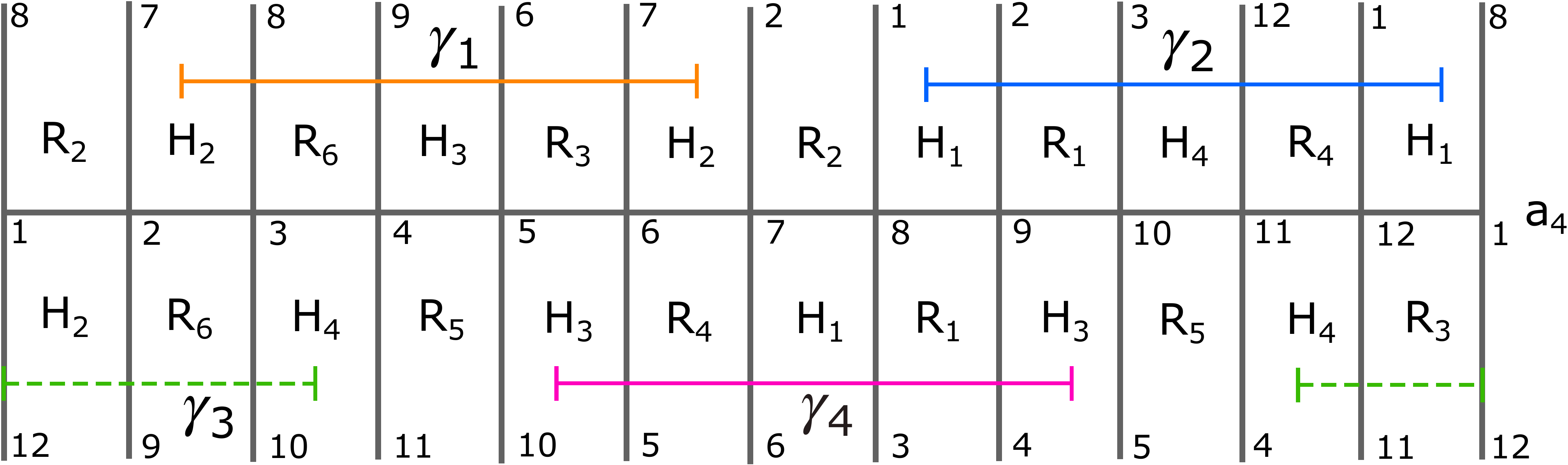}
\caption{Regular neighbourhood of $a_4$ with $a_4 \cap a_0$ marked as in figure \ref{fig:a0a4}. The vertical arcs represent $a_0$.}
\label{fig:BMM_a4}
\end{figure}

\begin{figure}
\centering
\includegraphics[width=\textwidth]{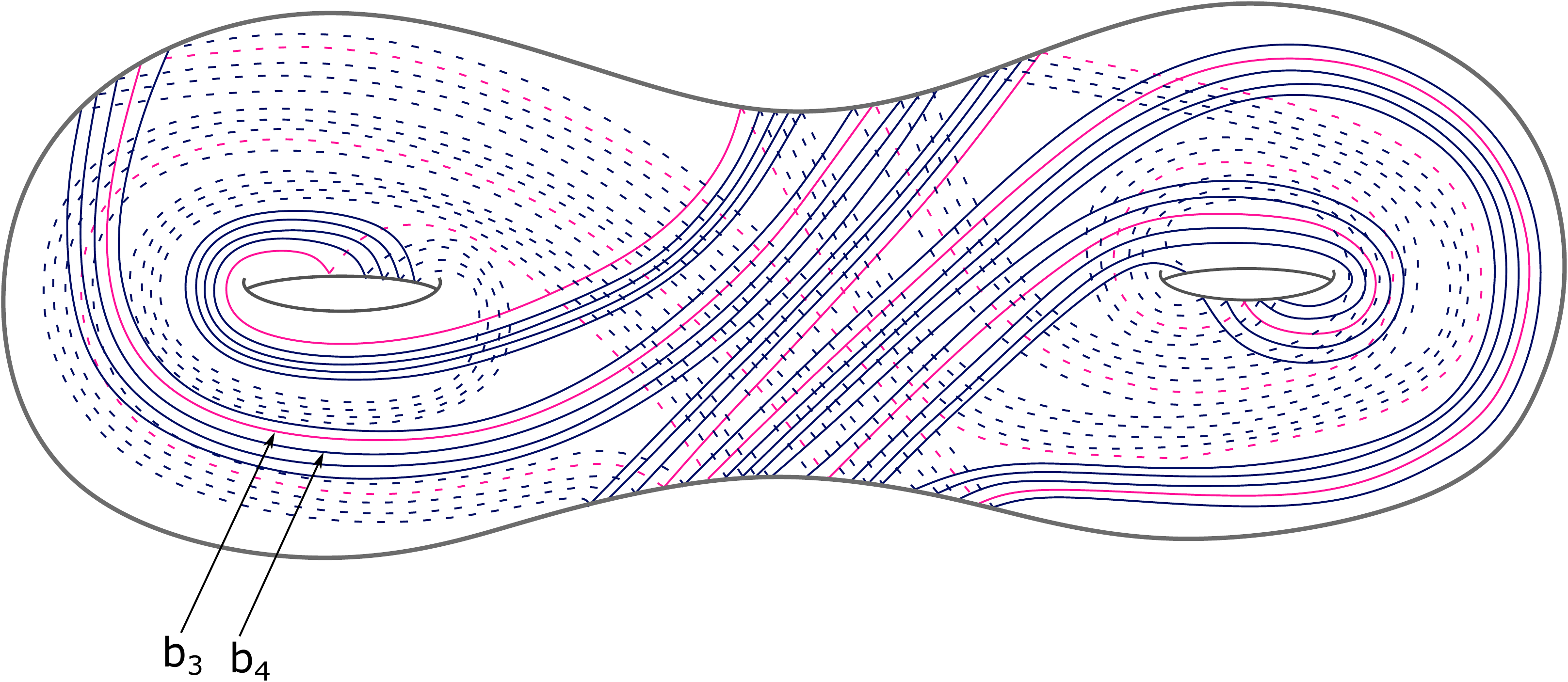}
\caption{}
\label{fig:b3b4}
\end{figure}

\begin{figure}
\centering
\includegraphics[scale=0.2]{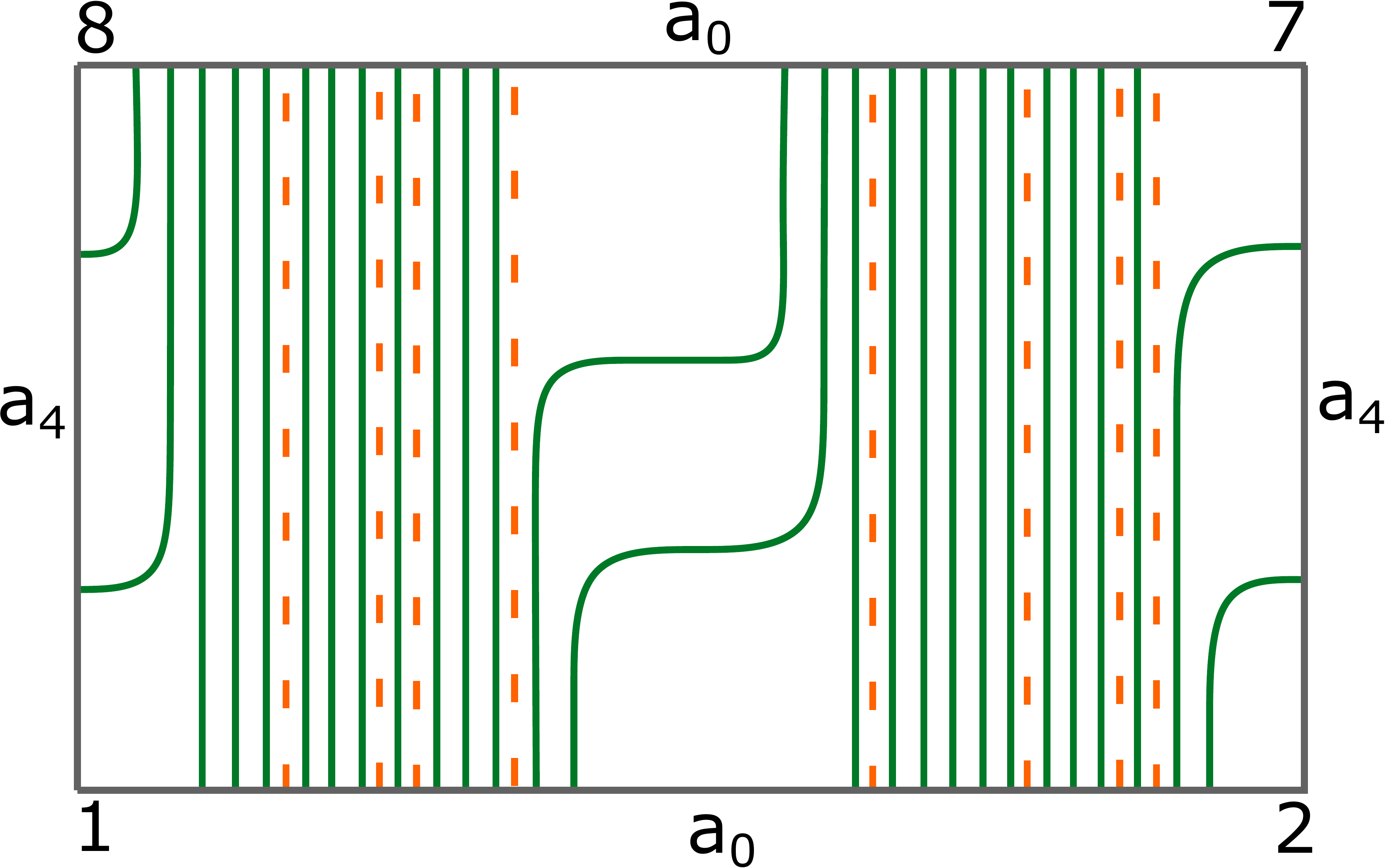}
\caption{$R_2$}
\label{fig:R2}
\end{figure}

\begin{figure}
     \centering
     \begin{subfigure}{0.45\textwidth}
         \centering
	\includegraphics[width=\textwidth]{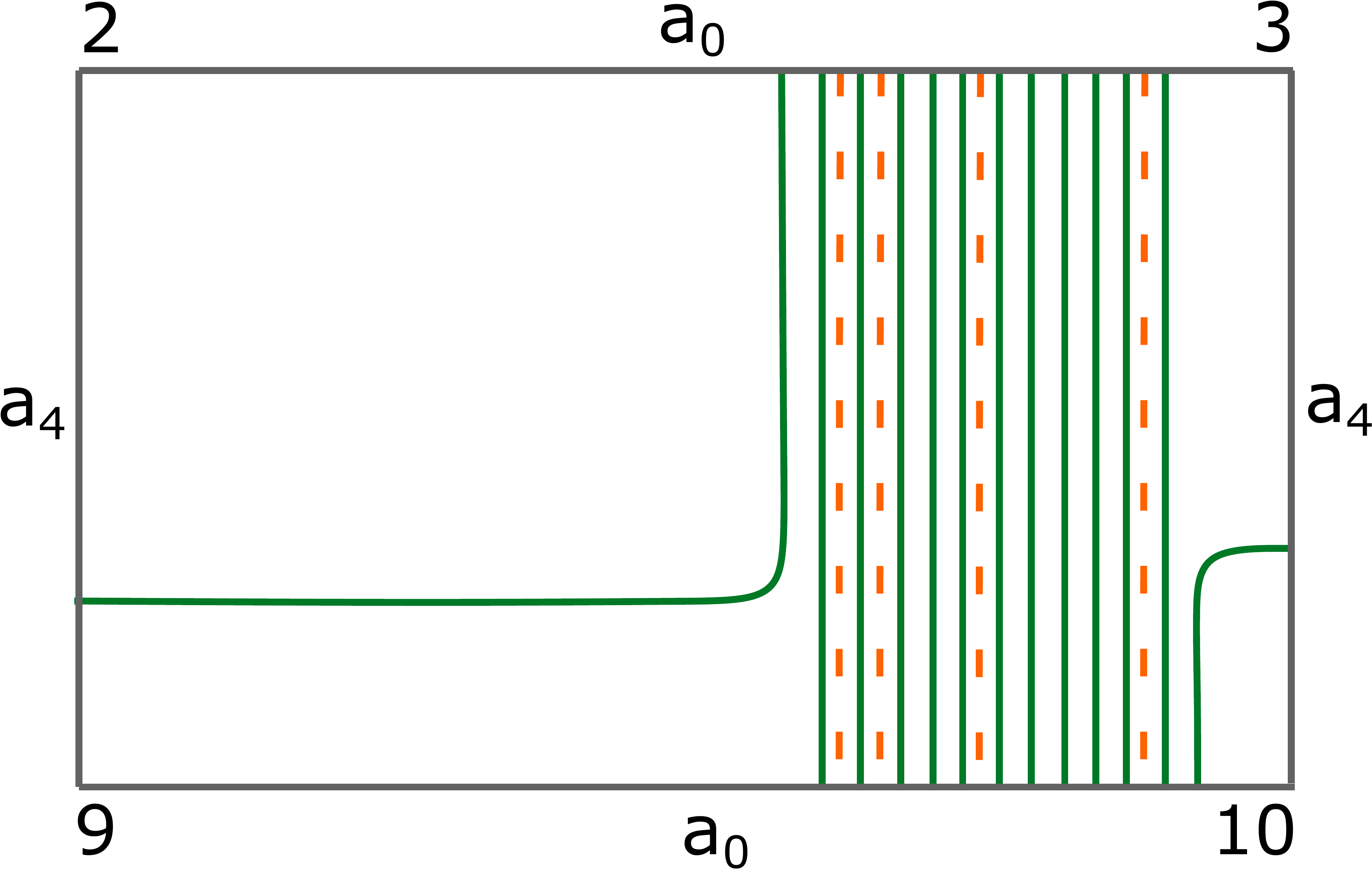}
	\caption{$R_1$}
	\label{fig:R1}
     \end{subfigure}
     \hfill
     \begin{subfigure}{0.45\textwidth}
         \centering
	\includegraphics[width=\textwidth]{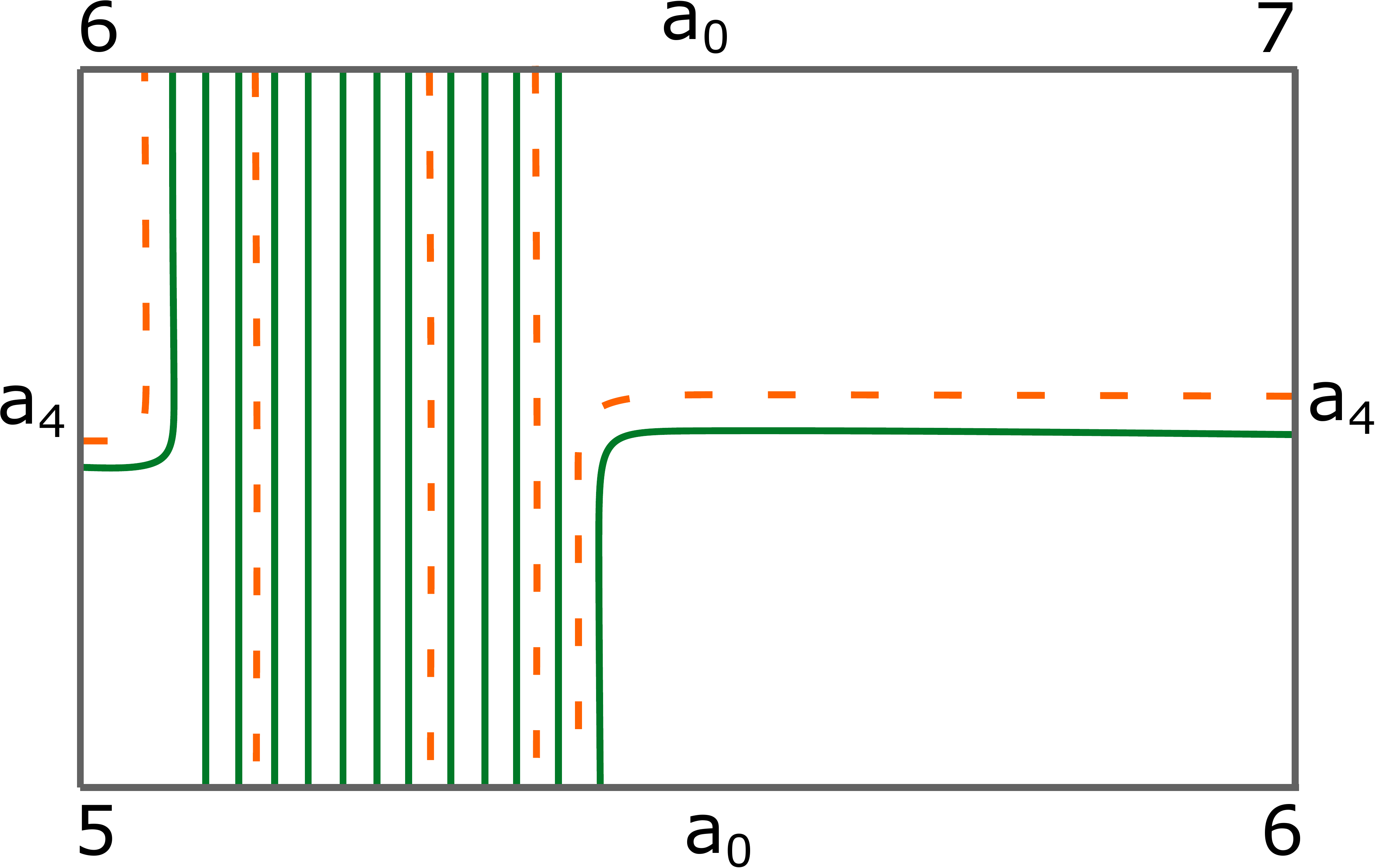}
	\caption{$R_3$}
	\label{fig:R3}
     \end{subfigure}
     \caption{}
\end{figure}

\begin{figure}
     \centering
     \begin{subfigure}{0.45\textwidth}
         \centering
	\includegraphics[width=\textwidth]{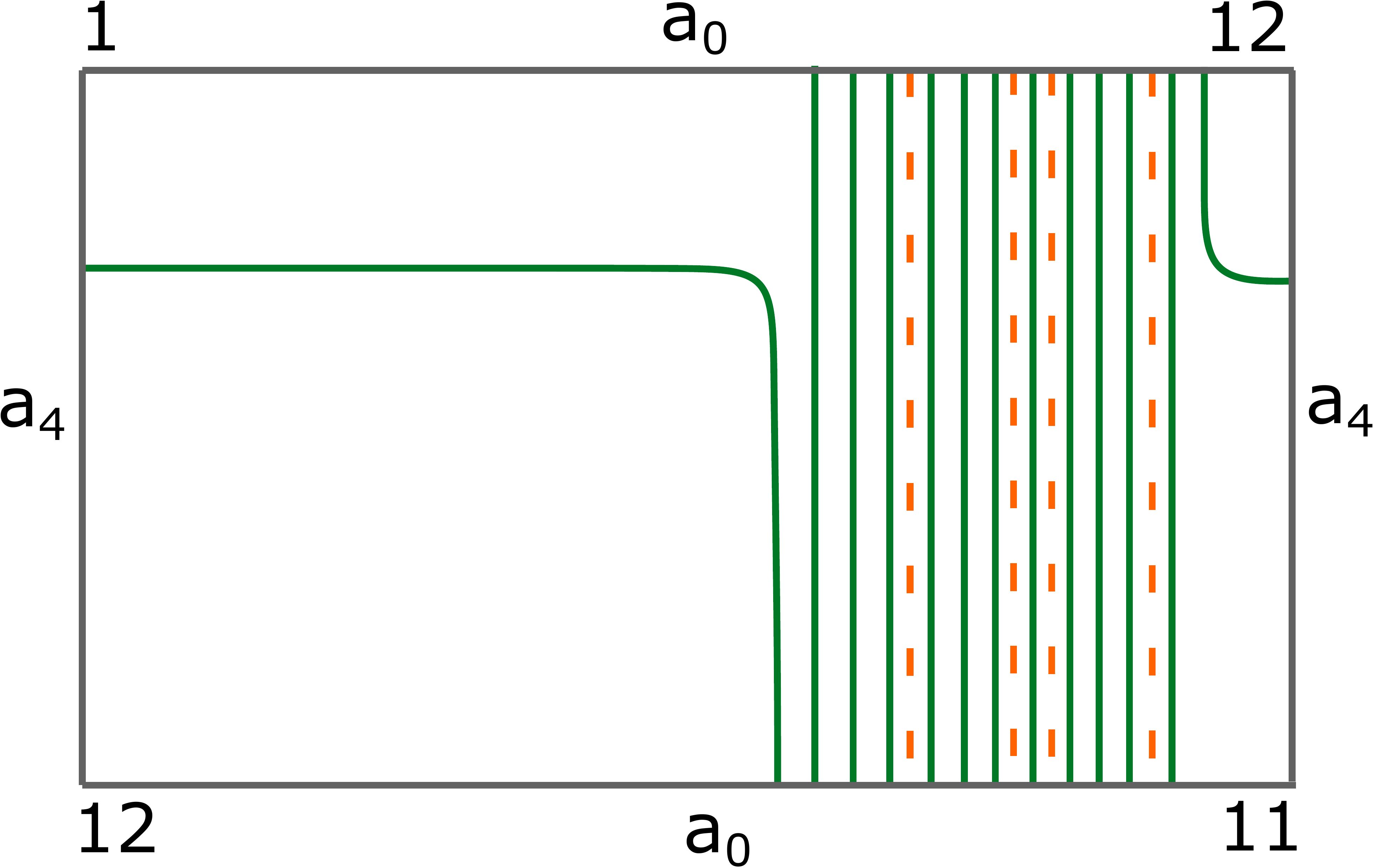}
	\caption{$R_4$}
	\label{fig:R4}
     \end{subfigure}
     \hfill
     \begin{subfigure}{0.45\textwidth}
         \centering
	\includegraphics[width=\textwidth]{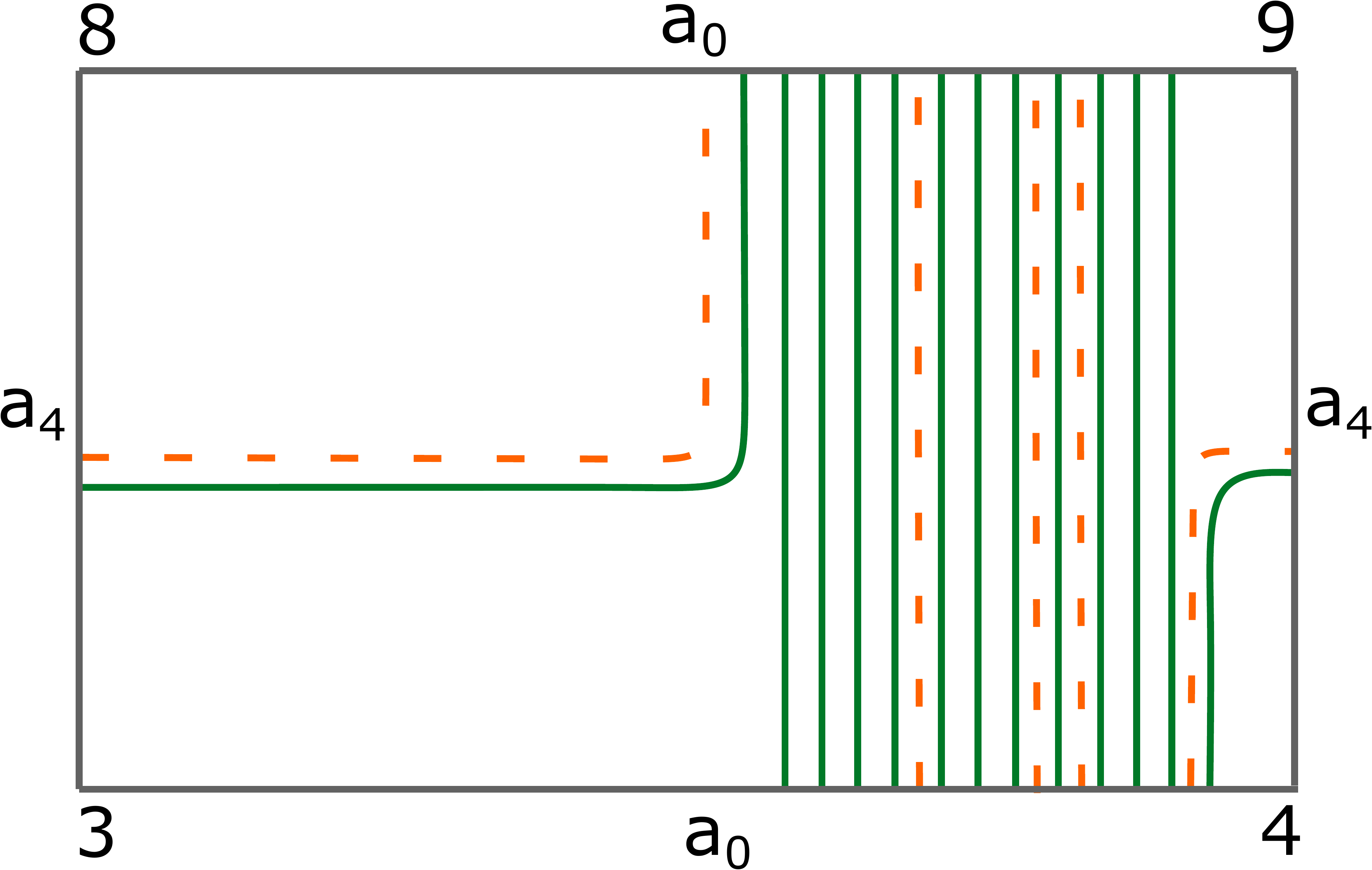}
	\caption{$R_6$}
	\label{fig:R6}
     \end{subfigure}
     \caption{}
\end{figure}

\begin{figure}
\centering
\includegraphics[scale=0.2]{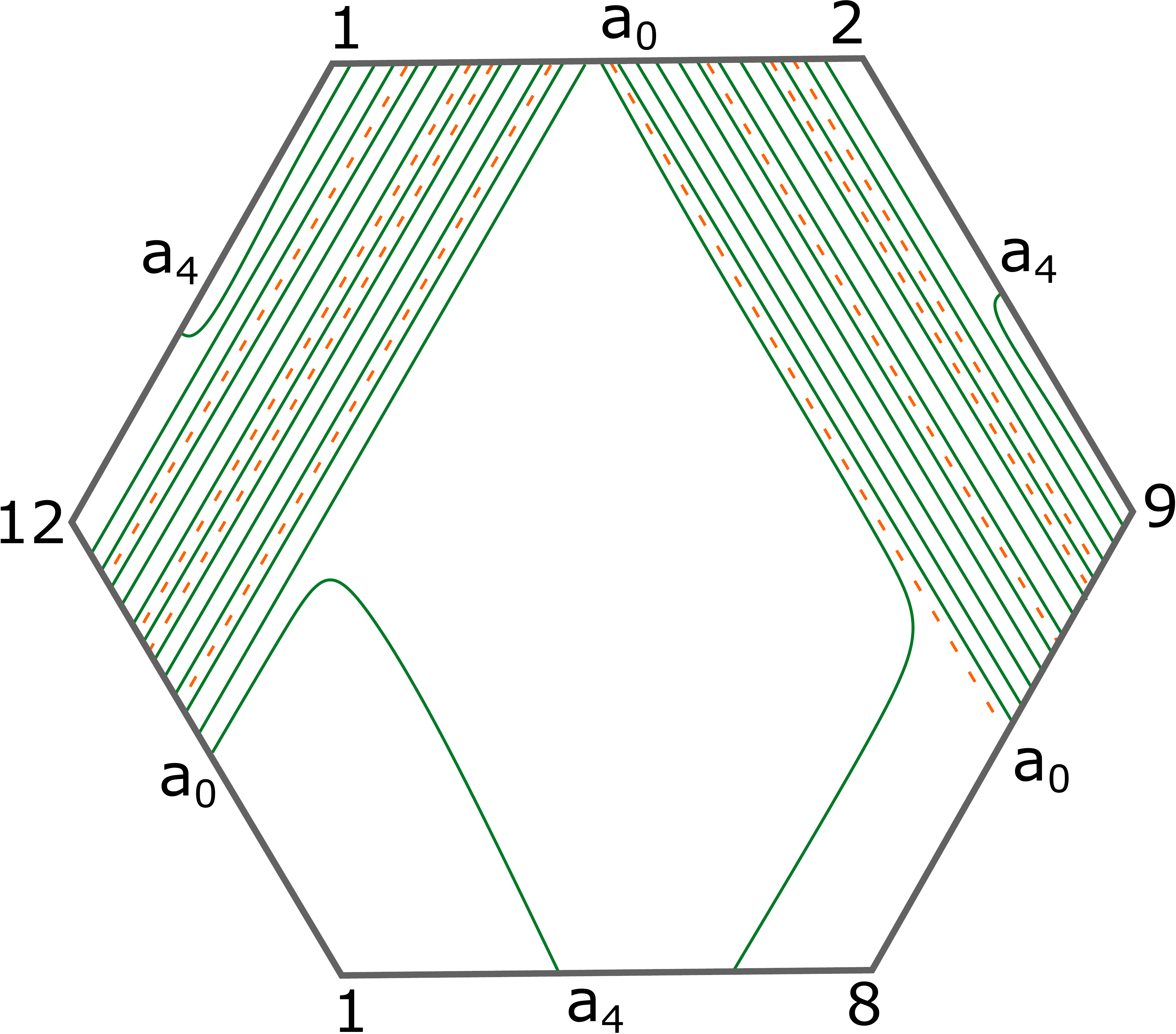}
\caption{$H_1$}
\label{fig:H1}
\end{figure}

\begin{figure}
\centering
\includegraphics[scale=0.2]{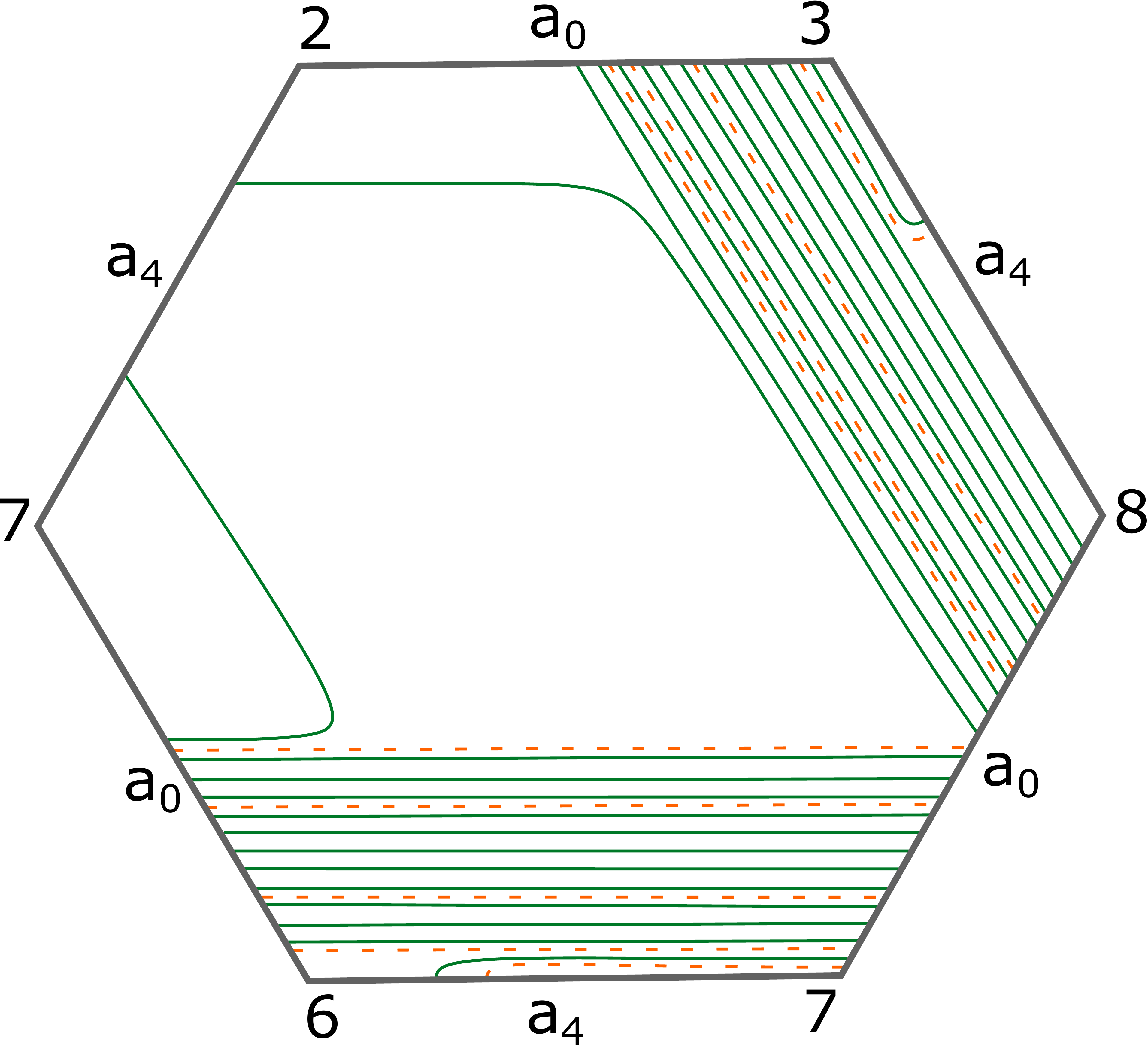}
\caption{$H_2$}
\label{fig:H2}
\end{figure}

\begin{figure}
     \centering
     \begin{subfigure}{0.45\textwidth}
         \centering
	\includegraphics[width=\textwidth]{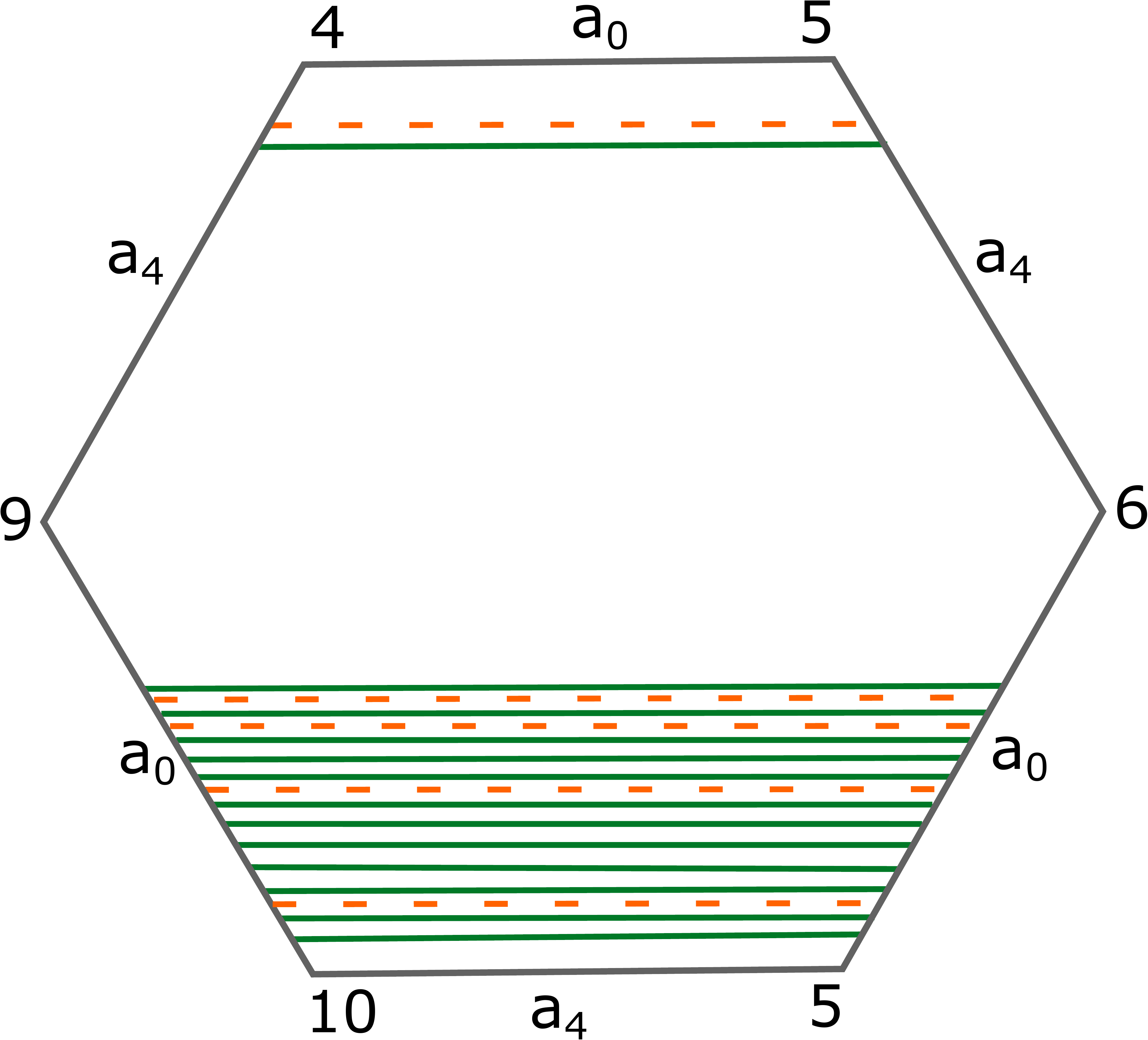}
	\caption{$H_3$}
	\label{fig:H3}
     \end{subfigure}
     \hfill
     \begin{subfigure}{0.45\textwidth}
         \centering
	\includegraphics[width=\textwidth]{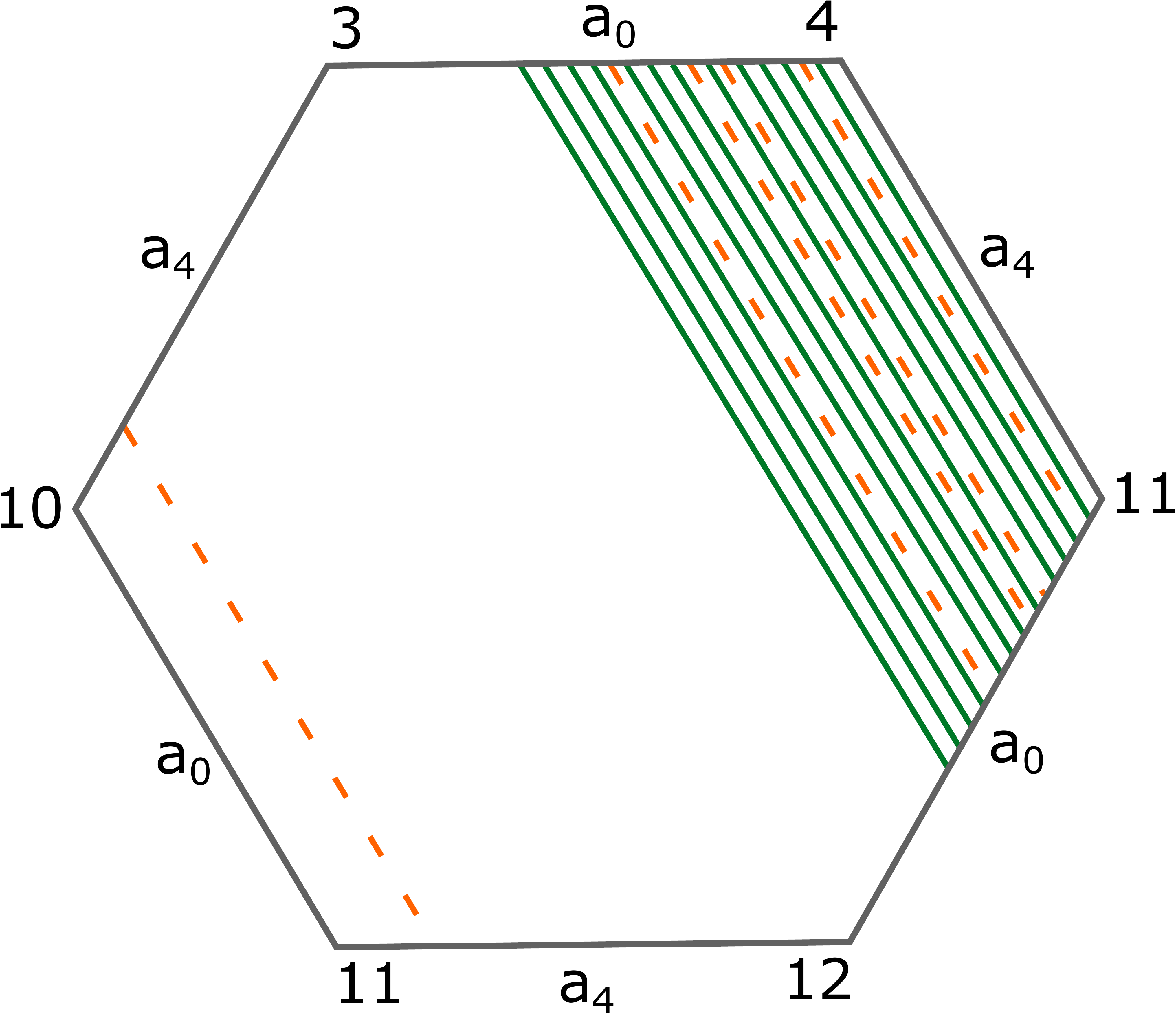}
	\caption{$H_4$}
	\label{fig:H4}
     \end{subfigure}
     \caption{}
     \label{fig:H3H4}
\end{figure}
\clearpage

\bibliographystyle{acm}
\bibliography{manuscriptbib}

\begin{thebibliography}{10}

\bibitem{A1}
{\sc Aougab, T.}
\newblock Uniform hyperbolicity of the graphs of curves.
\newblock {\em Geom. Topol. 17}, 5 (2013), 2855--2875.

\bibitem{AH1}
{\sc Aougab, T., and Huang, S.}
\newblock Minimally intersecting filling pairs on surfaces.
\newblock {\em Algebr. Geom. Topol. 15}, 2 (2015), 903--932.

\bibitem{AT1}
{\sc Aougab, T., and Taylor, S.~J.}
\newblock Small intersection numbers in the curve graph.
\newblock {\em Bull. Lond. Math. Soc. 46}, 5 (2014), 989--1002.

\bibitem{BMM}
{\sc Birman, J., Margalit, D., and Menasco, W.}
\newblock Efficient geodesics and an effective algorithm for distance in the
  complex of curves.
\newblock {\em Math. Ann. 366}, 3-4 (2016), 1253--1279.

\bibitem{B1}
{\sc Bowditch, B.~H.}
\newblock Uniform hyperbolicity of the curve graphs.
\newblock {\em Pacific J. Math. 269}, 2 (2014), 269--280.

\bibitem{CRS1}
{\sc Clay, M., Rafi, K., and Schleimer, S.}
\newblock Uniform hyperbolicity of the curve graph via surgery sequences.
\newblock {\em Algebr. Geom. Topol. 14}, 6 (2014), 3325--3344.

\bibitem{FM}
{\sc Farb, B., and Margalit, D.}
\newblock {\em A primer on mapping class groups}, vol.~49 of {\em Princeton
  Mathematical Series}.
\newblock Princeton University Press, Princeton, NJ, 2012.

\bibitem{H1}
{\sc Harvey, W.~J.}
\newblock Boundary structure of the modular group.
\newblock In {\em Riemann surfaces and related topics: {P}roceedings of the
  1978 {S}tony {B}rook {C}onference ({S}tate {U}niv. {N}ew {Y}ork, {S}tony
  {B}rook, {N}.{Y}., 1978)\/} (1981), vol.~97 of {\em Ann. of Math. Stud.},
  Princeton Univ. Press, Princeton, N.J., pp.~245--251.

\bibitem{HPW1}
{\sc Hensel, S., Przytycki, P., and Webb, R. C.~H.}
\newblock 1-slim triangles and uniform hyperbolicity for arc graphs and curve
  graphs.
\newblock {\em J. Eur. Math. Soc. (JEMS) 17}, 4 (2015), 755--762.

\bibitem{chicken}
{\sc Mahanta, K., and Palaparthi, S.}
\newblock Distance 4 curves on closed surfaces of arbitrary genus.
\newblock {\em Topology Appl. 314\/} (2022), Paper No. 108137.

\bibitem{MM1}
{\sc Masur, H.~A., and Minsky, Y.~N.}
\newblock Geometry of the complex of curves. {I}. {H}yperbolicity.
\newblock {\em Invent. Math. 138}, 1 (1999), 103--149.

\bibitem{MM2}
{\sc Masur, H.~A., and Minsky, Y.~N.}
\newblock Geometry of the complex of curves. {II}. {H}ierarchical structure.
\newblock {\em Geom. Funct. Anal. 10}, 4 (2000), 902--974.

\end{thebibliography}

\end{document}